\documentclass[twoside,11pt]{article}

\usepackage{blindtext}

\usepackage{jmlr2e}

\usepackage{bbm, upgreek, amssymb, verbatim, amsmath}
\usepackage{algorithmic, algorithm}
\usepackage{booktabs, subfig, caption}

\newcommand\Beq{\begin{eqnarray}}
\newcommand\Eeq{\end{eqnarray}}
\newcommand{\grad}{\nabla}
\newcommand{\var}[0]{\textrm{Var}}
\newcommand{\Var}[0]{\textrm{Var}}
\newcommand{\I}[0]{\textrm{I}}

\newtheorem{claim}[theorem]{Claim}

\usepackage{lastpage}
\jmlrheading{26}{2025}{1-\pageref{LastPage}}{9/23; Revised
12/24}{4/25}{23-1166}{James Chok and Geoffrey M. Vasil}
\ShortHeadings{title}{Chok and Vasil}

\ShortHeadings{Optimization Over a Probability Simplex}{Chok and Vasil}
\firstpageno{1}

\begin{document}

\title{Optimization Over a Probability Simplex}

\author{\name James Chok$^{1, 2}$ \email{james.chok@ed.ac.uk} \\
       \name Geoffrey M. Vasil$^{1}$ \email gvasil@ed.ac.uk \\
       \addr $^{1}$School of Mathematics and Maxwell Institute for Mathematical Sciences, The University of Edinburgh, 
       Edinburgh EH9 3FD, United Kingdom \\
       \addr $^{2}$Department of Applied Mathematics and Theoretical Physics, University of Cambridge, Wilberforce Road, Cambridge CB3 0WA,
        United Kingdom
        }

\editor{Silvia Villa}

\maketitle

\begin{abstract}
We propose a new iteration scheme, the Cauchy-Simplex, to optimize convex problems over the probability simplex $\{w\in\mathbb{R}^n\ |\ \sum_i w_i=1\ \textrm{and}\ w_i\geq0\}$.
Specifically, we map the simplex to the positive quadrant of a unit sphere, envisage gradient descent in latent variables, and map the result back in a way that only depends on the simplex variable. Moreover, proving rigorous convergence results in this formulation leads inherently to tools from information theory (e.g., cross-entropy and KL divergence). Each iteration of the Cauchy-Simplex consists of simple operations, making it well-suited for high-dimensional problems. In continuous time, we prove that $f(x_T)-f(x^*) = \mathcal{O}(1/T)$ for differentiable real-valued convex functions, where $T$ is the number of time steps and $w^*$ is the optimal solution. Numerical experiments of projection onto convex hulls show faster convergence than similar algorithms. Finally, we apply our algorithm to online learning problems and prove the convergence of the average regret for (1) Prediction with expert advice and (2) Universal Portfolios.
\end{abstract}

\begin{keywords}
  Constrained Optimization, Convex Hull, Simplex, Orthogonal Matrix, Gradient Flow
\end{keywords}
Accepted at JMLR.

\section{Introduction}
Optimization over the probability simplex, (\textit{i.e.}, unit simplex) occurs in many subject areas, including portfolio management \citep{up_helmbold, jmlr_portfolio_optimization, EGD_finance, canyakmaz2023}, machine learning \citep{ jmlr_ml_simplex_1, jmlr_ml_simplex_3, jmlr_ml_simplex_2, floto2023}, population dynamics \citep{CHA_population_dynamics, CHA_population_dynamics_and_genetics}, and multiple others including statistics and chemistry \citep{simplex_maximum_clique, complexity_of_optimization_over_simplex, CHA_chemistry, candelieri2025}.
In this paper, we look at minimizing differentiable real-valued convex functions $f(w)$ with $w\in\mathbb{R}^n$ within the probability simplex,
\begin{align}\label{eq:constrained_optimization_problem}
    \min_{w\in \Delta^n} f(w),\quad \textrm{where }\Delta^n\ =\ \{w\in\mathbb{R}^n\ | \ \sum_{i} w_{i} = 1\ \textrm{and} \ w_{i} \geq 0\}.
\end{align}

While quadratic programs can produce an exact solution under certain restrictions on $f$, they tend to be computationally expensive when $n$ is large \citep{numerical_optimization_jorge}. Here, we provide an overview of iterative algorithms to solve this problem approximately and introduce a new algorithm for differentiable real-valued convex functions $f$. We also show how our method can be applied to other constraints, particularly orthogonal matrices.

{Previous works (e.g., projected and exponentiated gradient descent) are equivalent to the discretizations of an underlying continuous-time gradient flow method. However, in each case, the gradient flow only enforces either the positivity or the unit-sum condition. Alternatively, the Frank-Wolfe algorithm is not derived from a gradient flow method and is inherently a discrete-time method.}

{We propose a continuous-time gradient flow that is able to satisfy both constraints. Being a continuous-time gradient flow allows us to prove convergence in a continuous-time setting, and we also prove convergence of its forward Euler discretization scheme. Moreover, we see the ideas of projected and exponentiated gradient descent in our algorithm.}

The remainder of this paper is structured as follows. In Section \ref{sec:previous_works}, we provided an overview of methods used to solve (\ref{eq:constrained_optimization_problem}). In Section \ref{sec:algorithm}, we present our proposed algorithm, show how it can be derived using simple calculus, and reveal connections to previous works. Theoretical analysis of the algorithm is shown in Section \ref{sec:convergence_proof}, proving a convergence rate of $\mathcal{O}(1/T)$. Section \ref{sec:orthogonal_matrix} shows how to extend the method presented in Section \ref{sec:algorithm} to constraints over orthogonal matrices. Finally, Section \ref{sec:applications} provides theoretical applications in game theory and finance, as well as numerical applications in geometry and finding how to weight survey questions to follow a desired distribution. 

\section{Previous Works}\label{sec:previous_works} 
{Two common classes of methods currently exist to find the minimum value of a smooth, convex function within a non-empty, convex, and compact \(R \subset \mathbb{R}^n \): (i) Projection-Based Methods, adjust guesses to stay in the set; and (ii) Frank-Wolfe Methods which reduce the function's value while explicitly staying in the set.}

{\textit{Projection-based methods} use a projection to satisfy the constraints of $R$.}

{This paper uses two definitions of a projection. Classically in the optimization literature \citep{Galntai2004}, projection of a point $y\in \mathbb{R}^n$ refers to the optimization procedure
\begin{align}\label{eq:def_projection}
    \textrm{proj}_R(y)\ =\ \min_{x\in R} \|x - y \|.
\end{align}
Since $R$ is convex and compact, a unique solution exists.
}

{Any linear idempotent map generalizes the classical notion of projection \citep{halmos1998}, \textit{i.e.}, $g:D\mapsto D$ for $D$ a non-empty set with $g(g(x))=g(x), \forall x\in D$. Defining the set $g(D) = \{g(x) | x \in D\} \subseteq D$, $g$ is said to project $D$ onto $g(D)$. This generalization of projection still encompasses the optimization procedure as if we take $g(y) = \textrm{proj}_R(y)$, then $g(g(y))=\textrm{proj}_R(g(y))=g(y)$.}

\textit{Projected gradient descent} (PGD) is a simple method to solve differentiable real-valued convex problems over $R$. The iteration scheme follows
\Beq\label{eq:original_pgd}
{w}^{t+1}\ =\ \textrm{proj}_R(w^t-\alpha_t \nabla_{w} f(w^t)),\qquad\text{where}\qquad \textrm{proj}_R(y)\ =\ \min_{x\in R} \|x - y\|,
\Eeq
and learning rate $\alpha_t > 0$. In general, $w^t - \alpha_t \nabla_{w} f(w^t)$ neither satisfies the positivity or unit-sum constraint. Projection of this vector into the unit simplex can be performed in $O(mn)$ operations \citep{michelot_1986, chen_2011, wang_2013, Kyrillidis_2013}, where $n$ is the dimension of $w$ and $1\leq m\leq n$ is the number of iterations of the algorithm. While this added cost is often negligible, it can become significant for large $n$. More recent applications increasingly require large dimensions, e.g. with $n \geq 1000 $ \citep{Markowitz_1994}.

This formulation of PGD can be simplified with linear constraints $Aw=b$, where $A$ is an $k\times n$ matrix and $b\in\mathbb{R}^k$. \cite{Luenberger_1997} projects $\nabla_w f(w^t)$ into the nullspace of $A$ and descends the result along the projected direction. For the unit-sum constraint, this algorithm requires solving the constrained optimization problem
\begin{equation*}
    \arg\min_{x} \frac{1}{2}\|\nabla_{w} f(w^t) - x\|^2, \quad \text{with}\quad \sum_i x_i\ =\ 0.
\end{equation*}
This problem yields to the method of Lagrange multipliers, giving the solution 
\begin{equation*}
    w^{t + 1}_i\ =\ w^t_i -\alpha_t\left(\nabla_{w_i} f(w^t) - \frac{1}{n}\sum_i\nabla_{w_i} f(w^t) \right).
\end{equation*}
While this scheme satisfies the unit-sum constraint similarly to (\ref{eq:original_pgd}), it does not satisfy the positivity constraint. This requires the same projection used in PGD, thus costing $\mathcal{O}(mn)$ operations \citep{numerical_optimization_jorge, sketching_method_convex_hull}.

\textit{Exponentiated Gradient Descent} (EGD) first presented by \cite{EGD_og_textbook}, and later by \cite{EGD_OG_linear_regression}, is a specific case of mirror descent. They consider gradient flow in the mirror space $(\Delta^n)^*$, and one maps between $\Delta^n$ and $(\Delta^n)^*$ via the functions $\nabla h:\Delta^n \mapsto (\Delta^n)^*$ and $(\nabla h)^{-1}:(\Delta^n)^*\mapsto \Delta^n$, where $h(w)=\sum_i (w_i\log(w_i) - w_i)$ \citep{Vishnoi_2021}. This yields the gradient flow in the mirror space
\begin{equation*}
    \frac{d\theta}{dt}\ =\ -\alpha\nabla_x f\big((\nabla h)^{-1}(\theta)\big),
\end{equation*}
for continuous learning rate $\alpha>0$. Mapping back into $\Delta^n$ yields the gradient flow
\begin{equation}\label{eq:continous_mirror_descent}
    \frac{d}{dt}\log(w_i)\ =\ -\alpha \nabla_{w_i} f(w),
\end{equation}
While this preserves positivity, the unit-sum constraint is not preserved. This can be seen as the differential equation can be rewritten as $dw_i/dt = -\alpha w_i\nabla_{w_i}f(w)$, and $\sum_i dw_i/dt$ is not zero in general.

A forward Euler discretization of (\ref{eq:continous_mirror_descent}) yields $w^{t+1}_i=w^t_i\exp(-\alpha_t \nabla_{w_i}f(w^t))$, thus mapping $\mathbb{R}^n_{\geq 0}\mapsto \mathbb{R}^n_{\geq 0}$ where $\mathbb{R}^n_{\geq 0}=\{x\in\mathbb{R}^n - \{0\} | x_i\geq 0\}$. 
To enforce the unit-sum constraint, one can use the projection defined in (\ref{eq:def_projection}) by taking $R=\Delta^n$; however, as suggested by \cite{EGD_OG_linear_regression}, a simple idempotent map to project $\mathbb{R}^n_{\geq 0}$ onto $\Delta^n$ can be used, given by
\begin{align*}
    w_i^{t+1}\ =\ \frac{w_i^t\exp(-\alpha_t \nabla_{w_i} f(w^t))}{\sum_{j} w_j^t\exp(-\alpha_t \nabla_{w_j}f(w^t))},
\end{align*}
with discrete learning-rate $\alpha_t>0$. {Thus the projection, given by the normalization factor, takes $\mathcal{O}(n)$ operations. Moreover, discretization is necessary for the algorithm to satisfy the constraint.}

\textit{Frank-Wolfe-based methods} exploit the convexity of the domain $R$, eliminating the need of a projection.

\textit{The Frank-Wolfe method} is a classic scheme \citep{FW_og} that is experiencing a recent surge popularity \citep{fw_svm_2, FW_distributed, FW_matrix_recovery, fw_svm_1}.
The method skips projection by assuming $R$ is convex. That is,
\begin{align}
& w^{t+1}\ =\ (1 - \gamma_t)w^t + \gamma_t s^t, \nonumber \\
& \textrm{where} \quad s^t\ =\ \arg \min_{s\in R}\ s\cdot \nabla_{w} f(w^t) \quad \textrm{and}\quad 0\ \leq\ \gamma_t\ \leq\ 1.\label{eq:Linear_Minimization_Oracle}
\end{align}
Since $R$ is convex, $w^{t+1} \in R$ automatically. For the simplex, the subproblem (\ref{eq:Linear_Minimization_Oracle}) is known in closed form.
Frank-Wolfe-based methods tend to be fast for sparse solutions but display oscillatory behavior near the solution, resulting in slow convergence \citep{FW_global_linear_convergence, FW_a_journey}.

\textit{The Pairwise Frank-Wolfe} (PFW) method improves upon the original by introducing an `away-step' to prevent oscillations allowing faster convergence \citep{FW_comments_away_steps, jaggi_2013, FW_global_linear_convergence}.

While PFW, EGD, and PGD have guaranteed convergence under constant and decreasing step sizes \citep{Vishnoi_2021, jaggi_2013}, a line search is often used in practice to improve run-time \citep{numerical_optimization_jorge}. Taking $f$ in (1) to be quadratic \citep{sqop_0, CHA_population_dynamics_and_genetics, selvi_2023}, a line search has analytical solutions for the Frank-Wolfe and PFW methods but not for EGD and PGD. EGD and PGD require approximate methods (e.g., backtracking method \citep{numerical_optimization_jorge}), adding extra run time per iteration.

\section{The Main Algorithm}\label{sec:algorithm}
For convex problems over a probability simplex, we propose what we named the Cauchy-Simplex (CS). 
{For an initial vector $w^0\ \in\ \mathrm{relint}(\Delta^n)\ =\ \{w\in\mathbb{R}^n| \sum_i w_i = 1, w_i > 0\}$, we propose the following iteration scheme}
\begin{align} \label{eq:discrete_time_scheme}
& w^{t+1}_i\ =\ w^t_i\ -\ \eta^t\, d^t_i,\quad \\
\notag & \text{with}\ d^t_i\ =\ w^t_i\, (\nabla_{w_i} f(w^t)\ -\ w^t\cdot \nabla_w f(w^t)),\quad \quad 0<\eta^t< {\eta}^{t,\max},\\
\notag &\text{and}\ \eta^{t, \max}\ =\ \frac{1}{\max_i(\nabla_{w_i} f(w^t)\ -\ w^t\cdot \nabla_w f(w^t))}.
\end{align}
\begin{remark}\label{remark:max_index}
    Rewriting 
    \begin{align}\label{eq:remark_max_index}
        \max_i\left(\nabla_{w_i}f(w^t)-w^t\cdot\nabla_w f(w^t)\right)\ =\ (e_j - w^t)\cdot \nabla f(w^t),
    \end{align}
    where $e_j\in\mathbb{R}^n$ is the standard unit vector with $j=\arg\max_i \nabla_{w_i}f(w^t)$. For finite iterations $T>0$, $w^T\in\textrm{relint}(\Delta^n)$ and thus $e_j-w^T\neq 0$. Therefore (\ref{eq:remark_max_index}) is zero when $\nabla f(w^T)=0$, $i.e.$, the optimal solution has been reached.
\end{remark}

When $\max_i(\nabla_{w_i} f(w^t) - w^t\cdot \nabla_w f(w^t))=0$, then $\eta^{t,\max}=\infty$ and $\eta^t$ is set to any finite positive constant. As seen in Remark \ref{remark:max_index}, this condition being met implies $\nabla f(w^t)=0$, and $w^{t+1}=w^t$ for finite $\eta^t$.

The upper bound on the learning rate, ${\eta}_{t}$, ensures that $w^{t+1}_i$ is positive for all $i$. Summing over the indices of $d^t$
\begin{align*}
    \sum_iw_i^t\, \Big(\nabla_{w_i}f(w^t)\, -\, w^t\cdot\nabla_w f(w^t)\Big)\ =\ (w^t \cdot \nabla_w f(w^t)) \bigg(1\ -\ \sum_i w_i^t \bigg).
\end{align*}

Defining the linear operator ${\mathcal{L}}(v) = \sum_i v_i$, if $\mathcal{L}(w^t)=1$, then $\mathcal{L}(d^t)=0$. This makes $w^{t+1}$ also satisfy the unit-sum constraint as $\mathcal{L}(w^{t+1})=\mathcal{L}(w^t)=1$. 

However, it is important to note that once an index is set to zero, it will stay zero for future iterations. Thus, having the initial condition in the relative interior of $\Delta^n$ ensures flow for each of the indices. Another important implication of this is that taking $\eta^t=\eta^{t,\max}$ may cause an index to be incorrectly set to zero. As such, step sizes should be taken such that $\eta^t<\eta^{t,\max}$ to prevent this.

\subsection{Motivating Derivation}
Our derivation begins by modeling $w\in\Delta^n$ through a latent variable, $\psi\in\mathbb{R}^n$,
\begin{align*}
    w_i\ =\ w_i(\psi)\ =\ \frac{\psi_i^2}{\sum_j\psi_j^2},
\end{align*}
which automatically satisfies positivity and unit probability. Thus, the optimization problem over the probability simplex can be considered as an unconstrained optimization problem
\begin{align*}
    \min_{\psi\in\mathbb{R}^n} F(\psi),\qquad\text{where}\qquad F(\psi)\ =\ f(w(\psi)).
\end{align*}
Now consider the continuous-time gradient descent on $\psi$,
\begin{equation}\label{eq:psi_gradient_flow}
    \frac{d \psi_j}{d t}\ =\ -\alpha \frac{\partial f}{\partial \psi_j},
\end{equation}
for continuous learning rate $\alpha>0$.
This then induces a gradient flow in $w$. To find this, first note that 
\begin{align*}
    \frac{\partial w_i}{\partial \psi_j}\ =\ \frac{2\psi_i \|\psi\|^2 \delta_{ij}\ -\ \psi_i^2\psi_j}{\|\psi\|^4}\ =\ \frac{2}{\|\psi\|^2}\left(\psi_i \delta_{ij}\ -\ 2 w_i \psi_j\right),
\end{align*}
where $\delta_{ij} = 1$ if $i = j$ and zero otherwise. Using the notation $\dot{\psi}=d\psi/dt$, the chain rule gives the gradient flow
\begin{align*}
    \frac{dw_i}{dt}\ =\ \sum_{j}\frac{\partial w_i}{\partial \psi_j}\frac{d\psi_j}{dt}\ =\ \frac{2}{\|\psi\|^2}\bigg(\psi_i \dot{\psi}_i\ -\ w_i \, \psi \cdot \dot{\psi} \bigg).
\end{align*}
By equation (\ref{eq:psi_gradient_flow}),
\begin{align*}
    \psi\cdot \dot{\psi}\ =\ -\alpha \sum_{i, j}\psi_j \frac{\partial w_i}{\partial \psi_j}\frac{\partial f}{\partial w_i}\ =\ -2\alpha (w\cdot \nabla_w f) \left(1 - \sum_j w_j\right)\ =\ 0.
\end{align*}

Thus, the gradient flow in $w$ can be simplified to
\begin{equation}\label{eq:cauchy_simplex_dw_dt}
\frac{dw_i}{dt}\ =\ \frac{2}{\|\psi\|^2} \, \psi_i \dot{\psi}_i\ =\ -\beta w_i (\nabla_{w_i} f(w) - w\cdot\nabla_w f(w))\quad\textrm{where}\quad \beta\ =\ \frac{4\alpha}{\|\psi\|^2}.
\end{equation}
Thus giving the iterative scheme
\begin{equation}\label{eq:cs_simplified}
w^{t+1}_i\ =\ w^t_i-\eta^t d^t_i \quad\text{where}\quad d^t_i\ =\ w^t_i \Big(\nabla_{w_i} f(w^t) - w^t\cdot\nabla_w f(w^t)\Big).
\end{equation}

While $w$ is initially modeled through a latent variable, the resulting gradient flow is only in terms of the constrained variable $w$. This allows the iteration scheme to only reference vectors in $\Delta^n$, and we no longer need to consider the latent domain $\mathbb{R}^n$. This derivation is to make a connection to recent work using a latent-variable approach \citep{lezcano_2019, Bucci_2022, li_2022}.

\subsection{On the Learning Rate}\label{sec:on_the_learning_rate}
The proof of convergence in Section \ref{sec:convergence_proof} assumes that $\eta^t < \eta^{t,\max}$; thus, all weights stay strictly positive but may be arbitrarily close to zero. However, accumulated rounding errors may result in some weights becoming zero or negative. As such, in our numerical implementation, a weight is set to zero once it is lower than some threshold (we choose 1e-10).

Once an index is set to zero, it will remain zero and can be ignored. This gives an altered maximum learning rate
\begin{equation*}
\tilde{\eta}^{t, \max}\ =\ \frac{1}{\max_{i\in S} (\nabla_{w_i} f(w^t) - w^t\cdot \nabla_w f(w^t))},
\end{equation*}
where $S=\{i \ | \ w_i > 0\}$. It follows that $\eta^{t, \max} \leq \tilde{\eta}_{t,\max}$, allowing for larger step-sizes to be taken.

The requirement of a maximum learning rate is not unique to the CS and is shared with PGD and EGD using fixed learning rates \citep{lu_2018}. However, this maximum learning rate is based on the Lipschitz constant of $\nabla f$, which is typically unknown. While numerical methods to approximate the Lipschitz constant exist, they can cause a failure of convergence \citep{Hansen_1992}.

More generally, PGD and EGD converge with a line search bounded by an arbitrarily large positive constant \citep{xiu_2007, Li_2018}, and becomes an extra parameter in the method. In contrast, each iteration of the CS has an easily computable maximum possible learning rate, $\eta^{t,\max}$, to be used in a line search.

\subsection{Connections to Previous Methods}
There are two ways of writing the gradient flow for the Cauchy-Simplex. In terms of the flow in $w$:
\begin{equation}\label{eq:cs_dw}
    \frac{dw}{dt}\ =\ -W\Pi_w \nabla_w f(w),\quad\text{where}\quad \Pi_w\ =\ \textrm{I} - \mathbbm{1}\otimes w,
\end{equation}
and $W$ is a diagonal matrix filled with $w$, $\textrm{I}$ is the identity matrix, and $\mathbbm{1}$ is a vector full of ones.

In terms of the flow in $\log(w)$:
\begin{equation}\label{eq:cs_dlog}
    \frac{d}{dt}\log(w)\ =\ -\Pi_w \nabla_w f(w),
\end{equation}
giving an alternative exponential iteration scheme 
\begin{equation}\label{eq:cs_discrete_exponential}
    w^{t+1}\ =\ w^{t}\exp(-\eta^t \Pi_w \nabla_w f(w^t)).
\end{equation}

\begin{claim}\label{claim:projection}
$\Pi_w$ projects $\mathbb{R}^n$ on the null space of the operator $\mathcal{L}_w(v)=\sum_i w_i v_i$, provided that $\sum_i w_i = 1$.
\end{claim}
\begin{proof}
    To see that $\Pi_w$ is an idempotent map, 
    \begin{align*}
    \Pi^2_w\ &=\ \I^2 - 2(\mathbbm{1} \otimes w) + (\mathbbm{1}\otimes w)(\mathbbm{1}\otimes w)  \nonumber \\ 
    &=\ \I - 2(\mathbbm{1} \otimes w) + \left(\sum_i w_i \right) \, (\mathbbm{1}\otimes w) \ = \ \Pi_w ,
    \end{align*}
    and is, therefore, a projection.
    
    To see that $\Pi_w(\mathbb{R}^n)$ is the null space of $\mathcal{L}_w$, it is easy to see that for any $u\in\mathbb{R}^n$, $\mathcal{L}_w(\Pi_w u)=0$. For the converse, let $v$ be in the null space of $\mathcal{L}_w$. By definition of the null space, $\mathcal{L}_w(v) = \sum_i w_iv_i = 0$, hence
    \begin{align*}
        \Pi_w(v)\ =\ (\I - \mathbbm{}1\otimes w)v\ =\ v.
    \end{align*}
    Thus $\Pi_w$ is surjective, mapping $\mathbb{R}^n\mapsto \mathrm{Nul}(\mathcal{L}_w)$.
\end{proof}

\begin{remark}
    {While $\Pi_w$ is a projection, it takes $\mathcal{O}(n)$ operations.}
\end{remark}

The formulations (\ref{eq:cs_dw}) and (\ref{eq:cs_dlog}) draw a direct parallel to both PGD and EGD, as summarized in Table \ref{tab:comparison_methods}. 

PGD can be written in continuous form as
\begin{equation*}
    \frac{dw}{dt}\ =\ - \Pi_{1/n}\nabla_w f(w) \quad\text{where}\quad \Pi_{1/n}\ =\ \I - \frac{1}{n}\mathbbm{1}\otimes \mathbbm{1} .
\end{equation*}
The projector helps PGD satisfy the unit-sum constraint, but not perfectly for general $w$. However, the multiplication with the matrix $W$ slows the gradient flow for a given index $w_i$ as it nears the boundary, with zero flow once it hits the boundary. Thus preserving positivity.

EGD, similarly, can be written in continuous form as
\begin{equation*}
    \frac{d}{dt}\log(w)\ =\ -\nabla_w f(w).
\end{equation*}
Performing the descent through $\log(w)$ helps EGD preserve positivity. Introducing the projector helps the resulting exponential iteration scheme (\ref{eq:cs_discrete_exponential}) to agree with its linear iteration scheme (\ref{eq:cs_simplified}) up to $\mathcal{O}((\eta^t)^2)$ terms. Thus helping preserve the unit-sum constraint.

\begin{claim}\label{clm:exponential_scheme_order_2}
    The Cauchy-Simplex exponentiated iteration scheme (\ref{eq:cs_discrete_exponential}) agrees up to $\mathcal{O}((\eta^t)^2)$ with it's linear iteration scheme (\ref{eq:discrete_time_scheme}). This can be seen by Taylor expanding (\ref{eq:cs_discrete_exponential}) w.r.t. $\eta^t$ around zero.
\end{claim}

\begin{remark}
    Combining PGD and EGD as
    \begin{align*}
        \frac{d}{dt}\log(w)\ =\ -\Pi_{1/n}\nabla_w f(w)
    \end{align*}
    preserves positivity but not the unit-sum constraint. This can be seen as the differential equation can be rewritten as $dw_i/dt=-w_i\Pi_{1/n}\nabla_{w_i} f(w)$. In general, $\sum_i dw_i/dt$ is non-zero and thus not an invariant property of the gradient flow.
\end{remark}

Unlike both PGD and EGD, the continuous-time dynamics of the CS are enough to enforce the probability-simplex constraint. This allows us to use the gradient flow of CS, \textit{i.e.} (\ref{eq:cauchy_simplex_dw_dt}), to prove convergence when optimizing convex functions (seen in Section \ref{sec:convergence_proof}). This contrasts with PGD and EGD, in which the continuous dynamics only satisfy one constraint. The discretization of these schemes is necessary to allow an additional projection step, thus satisfying both constraints.

\begin{table}[]
\caption{{Comparison of gradient flow for different optimization methods and constraints.}}
\label{tab:comparison_methods}
\centering
\renewcommand{\arraystretch}{2}
\begin{tabular}{|c|c|c|}
     \cline{2-3}
     \multicolumn{1}{c|}{} & $\sum_i w_i \neq 1$ & $\sum_i w_i = 1$\\
     \hline
     $w_i \ngeq 0$ & GD: $\frac{dw}{dt} = -\nabla_w f(w)$ & (\citeauthor{Luenberger_1997}) PGD: $\frac{dw}{dt} = -\Pi_{1/n}\nabla_w f(w)$\\ 
     \hline
     $w_i \geq 0$  &  EGD: $\frac{d}{dt}\log(w) = -\nabla_w f(w)$ & CS: $\frac{d}{dt}\log(w) = - \Pi_w \nabla_w f(w)$\\
     \hline    
\end{tabular}
\end{table}

\subsection{The Algorithm}
The pseudo-code of our method can be seen in Algorithm \ref{alg:thats_cap}.

\begin{algorithm}
\caption{Our proposed algorithm}\label{alg:thats_cap}
\begin{algorithmic}
\REQUIRE $\varepsilon \gets 10^{-10}$ \quad (Tolerance for the zero set)
\STATE{$w^0 \gets(1/n,\cdots, 1/n)$}
\WHILE{termination conditions not met}
    \STATE{$S \gets \{i = 1,\cdots, n\ | \ w_i > \varepsilon \}$} 
    \STATE{$Q \gets \{i = 1,\cdots, n\ | \ w_i \leq \varepsilon \}$}
    \STATE{}
    \STATE{{Choose $\eta^t > 0$}} 
    \STATE{$\eta^{\max} \gets \cfrac{1}{\max_{i \in S} (\nabla_{w_i} f(w^t)) - w^t \cdot\nabla_w f(w^t)}$}
    \STATE{$\eta^t \gets \min({\eta}^t, \eta^{\max})$}
    \STATE{}
    \STATE{$\hat{w}^{t + 1}_i \gets w^t_i -\eta^t w^t_i(\nabla_{w_i} f(w^t) - w^t\cdot\nabla_w f(w^t))$}
    \STATE{}
    \STATE{$\hat{w}_j^{t + 1} \gets 0, \ \ \ \forall j \in Q$}
    \STATE{$w_i^{t + 1} \gets \hat{w}_i^{t + 1} / \sum_j\hat{w}_j^{t + 1}$ \qquad \qquad (Normalizing for numerical stability)}
\ENDWHILE
\end{algorithmic}
\end{algorithm}

\section{Convergence Proof}\label{sec:convergence_proof}
We prove the convergence of the Cauchy-Simplex via its gradient flow. We also state the theorems for convergence of the discrete linear scheme but leave the proof in the appendix.

\begin{theorem}\label{thm:decreasing}
Let $f$ be convex with Lipschitz continuous gradient, real-valued and continuously differentiable w.r.t. $w^t$ and $w^t$ continuously differentiable w.r.t. $t$. For the Cauchy-Simplex gradient flow (\ref{eq:cauchy_simplex_dw_dt}) with initial condition $w^0\in\textrm{relint}(\Delta^n)$, $f(w^t)$ is a strictly decreasing function and stationary at the optimal solution for all $t\in[0,T]$ and finite $T>0$.
\end{theorem}
\begin{proof}
Notice that $dw^t/dt$ can be rewritten as
\begin{align*}
    \frac{d}{dt}\log(w^t_i)\ =\ -\Pi_{w^t}\nabla_w f(w^t).
\end{align*}
Since $\nabla f$ is Lipschitz, and $\Delta^n$ is a compact subset of $\mathbb{R}^n$, $\|\nabla f(w^t)\|$ is bounded for all $w^t\in\Delta^n$. Thus, if $w^0\in\textrm{relint}(\Delta^n)$, strict positivity of $w_i^t$ is preserved for $t\in[0,T]$. Furthermore, since $\sum_i dw_i/dt=\sum_i (W\Pi_w \nabla_w f(w))_i=0$, $\sum_i w_i$ is an invariant quantity of the gradient flow. Therefore, if $w^0\in\textrm{relint}(\Delta^n)$ then $w^t\in\textrm{relint}(\Delta^n)$ for $t\in [0, T]$, and finite $T>0$.

By direction computation
\begin{align*}
    \frac{df}{dt}\ &=\ \frac{\partial f}{\partial w^t}\cdot\frac{dw^t}{dt}\ =\ -\sum_i \nabla_{w_i} f(w^t) \bigg(w^t_i (\nabla_{w_i} f(w^t) - w^t\cdot\nabla_w f(w^t))\bigg)\\
    &=\ -\bigg(\sum_i w_i^t\left(\nabla_{w_i} f(w^t)\right)^2 - \left(w^t\cdot \nabla_w f(w^t)\right)^2\bigg) \\
    &:=\ -\var[\nabla_w f(w^t)\, |\, w^t],
\end{align*}
where $\var[v\, |\,w]$ is the variance of a vector $v\in\mathbb{R}^n$ with respect to the discrete measure $w\in\Delta^n$. 
The variance can be rewritten as
\begin{equation*}
    \var[v\, |\, w]\ =\ \sum_i w_i (v_i - v\cdot w)^2\ =\ \sum_i w_i(\Pi_w v)_i^2,
\end{equation*}
for all $v\in\mathbb{R}^n$, and thus is non-negative. Since $w^t\in\textrm{relint}(\Delta^n)$, it follows that $\frac{df}{dt}\leq 0$ and that $f$ is decreasing in time.

As $w^t\in\textrm{relint}(\Delta^n)$, $df/dt=0$ only when $\Pi_{w_t} \nabla_w f(w^t)=0$. However, this occurs only if $w^t$ is an optimal solution to (\ref{eq:constrained_optimization_problem}), which can be verified by checking the KKT conditions with the constraints of the simplex, \textit{i.e.}, $\sum_i w_i=1$, and $w_i \geq 0$, shown in Appendix \ref{sec:KKT}.

Thus $f$ is strictly decreasing in time, and stationary only at the optimal solution.
\end{proof}

\begin{theorem}
Let $f$ be real-valued, convex and continuously differentiable w.r.t. $w^t$, $w^t$ continuously differentiable w.r.t. $\,t$, and $w^*\in\Delta^n$ be a solution to (\ref{eq:constrained_optimization_problem}). Under the Cauchy-Simplex gradient flow (\ref{eq:cauchy_simplex_dw_dt}), for all $t\in[0,T]$ and finite $T>0$, the relative entropy
\begin{equation*}
    D(w^*|w^t)\ =\ \sum_{w^*_i\neq 0}w^*_i\log\left(\frac{w^*_i}{w^t_i}\right)
\end{equation*}
is a decreasing function in time for $w^t\in \text{relint}(\Delta^n)$.
\end{theorem}

\begin{proof}
We rewrite the relative entropy into
\begin{equation*}
    D(w^*|w^t)\ =\ \sum_{w^*_i\neq 0}w^*_i\log(w^*_i) - \sum_{i} w^*_i\log(w_i^t).
\end{equation*}
By direction computation
\begin{align*}
    \frac{d}{dt}D(w^*|w^t)\ &=\ -\sum_i\frac{w^*_i}{w_i^t}\frac{d w_i^t}{dt}\ =\ \sum_i w^*_i(\Pi_{w^t} \grad_{w_i} f(w^t))\ =\ \nabla_{w} f(w^t) \cdot(w^* - w^t).
\end{align*}
Since $f$ is convex, and $w^*$ is a minimum of $f$ in the simplex,
\begin{equation}\label{eq:theorem_2_ans}
    \frac{d}{dt}D(w^*|w^t)\ \leq\ f(w^*) - f(w^t) \leq 0.
\end{equation}
\end{proof}

\begin{theorem}
{Let $f$ be real-valued, convex and continuously differentiable w.r.t. $w^t$, $w^t$ continuously differentiable w.r.t. $t$, and $w^*\in \Delta^n$ a solution to (\ref{eq:constrained_optimization_problem}).} For $w^0=(1/n, \ldots, 1/n)$ and finite $T>0$, the Cauchy-Simplex gradient flow (\ref{eq:cauchy_simplex_dw_dt}) gives the bound
\begin{align*}
    f(w^T) - f(w^*)\ \leq\ \frac{\log(n)}{T}.
\end{align*}
\end{theorem}

\begin{proof}
Taking $u = w^*$ and integrating (\ref{eq:theorem_2_ans}) w.r.t. $t$ gives
\begin{equation*}
    \int^T_0\bigg(f(w^t) - f(w^*)\bigg)\, dt\ \leq\ -\int^T_0 \frac{d}{dt}D(w^*|w^t)\, dt\ =\ D(w^* | w^0) - D(w^* | w^T).
\end{equation*}

By Jensen's inequality, the relative entropy can be bounded from below by
\begin{align*}
    D(u|v)\ =\ - \sum_{u_i\neq 0} u_i\log(v_i / u_i)\ \geq\ - \log\left(\sum_{u_i\neq 0} v_i \right)\ \geq\ 0,
\end{align*}
for $u\in\Delta^n$ and $v\in\text{relint}(\Delta^n)$. This can also be shown using the Bergman divergence \citep{chou_2023}.

Thus, relative entropy is non-negative \citep{jaynes_probability, gibbs_statistical_mechanics}, yielding
\begin{equation*}
    \int^T_0\bigg(f(w^t) - f(w^*)\bigg)\, dt\  \leq\ D(w^* | w^0).
\end{equation*}
Completing the integral on the left side of the inequality and dividing by $T$ gives
\begin{equation*}
    \frac{1}{T}\int^T_0 f(w^t)\, dt - f(w^*)\ \leq\ \frac{1}{T}D(w^*|w^0).
\end{equation*}
Using Theorem \ref{thm:decreasing}, $f(w^t)$ is a decreasing function. Thus 
\begin{equation*}
    f(w^T)\ =\ \frac{1}{T}\int^T_0 f(w^T)\, dt\ \leq\ \frac{1}{T}\int^T_0 f(w^t)\, dt.
\end{equation*}

{Let $w^0=(1/n,\ldots, 1/n)$, then the relative entropy can be bounded by 
\footnote{{Intuitively, this can be seen by computing the relative entropy between a state with maximal entropy (the uniform distribution) and minimal entropy (the Dirac mass).}}
\begin{align*}
    D(u|w^0)\ =\ \sum_{u_i\neq 0} u_i \log(u_i)\, +\, \log(n)\ \leq\ \log(n),\qquad\text{for all}\qquad u\in\Delta^n.
\end{align*}
}
This gives the required bound
\begin{equation*}
    f(w^T) - f(w^*)\ \leq\ \frac{D(w^*|w^0)}{T} \leq \frac{\log(n)}{T}.
\end{equation*}
\end{proof}

\begin{theorem}[Convergence of Linear Scheme]\label{thm:convergence_of_linear_scheme}
Let $f$ be a differentiable convex function that obtains a minimum at $w^*$ with $\nabla f$ $L$-Lipschitz continuous. Let $w^0 = (1/n,\ldots, 1/n)$ and $\{\eta^t\}_{t=0}^{T-1}$ be a sequence that satisfies $0<\eta^t < \min\{1/L, \eta^{t,\max}\}$ and \footnote{Note that $C_{\gamma_t}$ is an increasing function of $\gamma_t$, with $C_{\gamma_t}\geq0$ for $\gamma_t\in[0, 1]$. Thus $\gamma_t=\eta^t/\eta^{t,\max}$ can also be chosen to satisfy (\ref{eq:appendix_c_gamma_inequality}) and that $\{\eta^t\}_t$ is a sequence with $0< \eta^t < \min\{1/L, \eta^{t,\max}\}$.}
\begin{align}\label{eq:appendix_c_gamma_inequality}
    C_{\gamma_t}\ \leq\ \frac{w^t\cdot (\Pi^t \nabla_w f(w_t))^2}{2 \max_i (\Pi^t \nabla_w f(w_t))^2_i},
\end{align}
with $\gamma_t = \eta^t/\eta^{t,\max}$, $C_{\gamma_t} = \gamma_t^{-2}\log(e^{-\gamma_t}/(1 - \gamma_t))$ and $\eta^{t,\max}$ defined in (\ref{eq:discrete_time_scheme}). Then the linear Cauchy-Simplex scheme (\ref{eq:discrete_time_scheme}) produces iterates $\{w^t\}_{t=0}^{T-1}$ such that, for finite integer $T > 0$,
\begin{align*}
    f(w^T) - f(w^*)\ \leq\ \frac{\log(n)}{\sum_{t=0}^{T-1}\eta^t}.
\end{align*}
\end{theorem}
The proof can be found in Appendix \ref{sec:appendix_discrete_proof}.

In practice, finding $\eta_t$, and by extension $\gamma_t$, to satisfy the assumptions of Theorem \ref{thm:convergence_of_linear_scheme} can be hard. Instead, we show asymptotic convergence of the linear Cauchy-Simplex scheme under a line search. 
\begin{lemma}[Asymptotic Convergence of Linear Scheme]\label{lemma:asymptotic_convergence}
    Let $f$ be a differentiable convex function with $\nabla f$ Lipschitz continuous. The linear Cauchy-Simplex (\ref{eq:discrete_time_scheme}) has asymptotic convergence when $\eta_t$ is chosen through a line search, \textit{i.e.},
    \begin{align*}
        f(w^{t+1})\ =\ \min_{\eta^t\in [0, \eta^{t,\max}-\varepsilon]}f(w^t -\eta^t d^t),
    \end{align*}
    for some $0<\varepsilon\ll 1$ and $d^t_i=w_t^t(\Pi^t\nabla f^t)_i$. That is $f(w^t)\to f(w^*)$ as $t\to\infty$.
\end{lemma}
The proof can be found in Appendix \ref{sec:asymptotic_convergence}.

\section{Extension: Optimization over Orthogonal Matrices}\label{sec:orthogonal_matrix}
Another often explored constraint is the orthogonal matrix constraint 
\begin{align*}
    \min_{Q \in \mathcal{S}^n} f(Q),\quad\text{where}\quad \mathcal{S}^n\ =\ \{Q \in \mathbb{R}^{n \times n}\ |\ Q Q^T\ =\ I\},
\end{align*}
with $\mathcal{S}^n$ also known as the Stiefel manifold. To name a few, this problem occurs in blind source separation \citep{source_separation_1, source_separation_2}, principle component analysis \citep{sparse_pca}, and neural networks \citep{orthogonal_nn_2, orthogonal_nn, Woodworth_2020, chou2024robust}.

Various iterative methods have been suggested to solve this problem, with them split up between Riemannian optimization, which uses expensive iterations that remain inside the Stiefel manifold \citep{orthogonal_matrix_retraction, orthogonal_matrix_retraction_2}, and landing methods \citep{orthogonal_matrix_landing}, which use cheap iterations that are not in the Stiefel manifold but over time will be arbitrarily close to the manifold. 

Using a similar method as in Section \ref{sec:algorithm}, we can also derive an explicit scheme that preserves orthogonality up to an arbitrary accuracy. 

Let $X\in\mathbb{R}^{n\times n}$ be an anti-symmteric matrix ($X = -X^T$). Then the Cayley transform
\begin{align*}
    Q\ =\ (I - X)(I + X)^{-1}\ =\ 2(I - X)^{-1} - I,
\end{align*}
parameterizes orthogonal matrices with $\text{Det}(Q) = 1$, where $I$ is the identity matrix. Performing similar calculations as above yields the gradient flow
\begin{align*}
    \frac{dQ}{dt}\ =\ -\eta\, \Omega(\Lambda \Omega^T - \Omega \Lambda^T)\Omega
\end{align*}
where $\Omega = Q + I$, and $\Lambda = \Omega^T\, \partial_Q f$ (full derivation can be found in Appendix \ref{sec:cayley_transform_derivation}).

However, an Euler discretization of this differential equation does not produce a scheme that preserves orthogonality. Instead, we consider a corrected iteration scheme
\begin{align*}
    Q_{t+1}\ =\ (I + C)(Q_t - \eta dQ)\quad\text{where}\quad dQ\ =\ \Omega(\Lambda \Omega^T - \Omega \Lambda^T)\Omega,
\end{align*}
for some matrix $C$. An iteration scheme that preserves orthogonality up to arbitrary accuracy can then be made by looking at the coefficients for powers of $\eta$ in the expansion of $Q_{t+1}Q_{t+1}^T$ and solving for $C$. In particular, a $\mathcal{O}(\eta^2)$ correct scheme can be written as
\begin{align*}
    Q_{t+1}\ =\ (I - \frac{\eta^2}{2} dQ\, dQ^T)(Q_t - \eta dQ),
\end{align*}
and a $\mathcal{O}(\eta^4)$ correct scheme is
\begin{align*}
    Q_{t+1}\ =\ (I - \frac{\eta^2}{2} dQ\, dQ^T + \frac{3\eta^4}{8}dQ\, dQ^T\, dQ\, dQ^T)(Q_t - \eta dQ).
\end{align*}

As such, parameterizing the manifold by a latent Euclidean space yields a gradient flow that remains explicitly inside the Stiefel manifold. Moreover, this gradient flow is explicitly in terms of the orthogonal matrix. This contrasts with Riemannian optimization methods, which use Riemannian gradient flow on the manifold. Moreover, since the tangent space at each point of the Stiefel manifold does not lie in the manifold itself, once the scheme is discretized, an exponential map (or retraction) must be used to remain on the manifold. These often involve costly matrix operations like exponentials, square roots, and inversions \citep{orthogonal_matrix_retraction_2, orthogonal_matrix_square_root_rectraction, lezcano_2019}. In contrast, working in Euclidean space allows the discretized scheme to remain on the manifold using cheaper addition and multiplication matrix operations. 

\section{Applications}\label{sec:applications}
As noted in Section \ref{sec:on_the_learning_rate}, approximations of the Lipschitz constant $L$ can cause failure of convergence. As such, our experiments use a line search outlined in Lemma \ref{lemma:asymptotic_convergence}. 

\subsection{Projection onto the Convex Hull}
Projection onto a convex hull arises in many areas like machine learning \citep{jmlr_convex_projection_2, jmlr_convex_projection_3, jmlr_convex_projection}, collision detection \citep{CHA_collision_detection} and imaging \citep{CHA_imaging_2, jmlr_convex_imaging}. It involves finding a point in the convex hull of a set of points $\{x_i\}_{1\leq i\leq n}$, with $x_i\in\mathbb{R}^d$, that is closest to an arbitrary point $y\in\mathbb{R}^d$, \textit{i.e.},
\begin{equation*}
\min_w \|wX - y\| ^ 2 \qquad \textrm{where} \qquad \sum_iw_i\ =\ 1\ \textrm{and}\ w_i\ \geq\ 0, 
\end{equation*}
and $X=[x_1, \cdots, x_n]^T$ is a $\mathbb{R}^{n\times d}$ matrix. This is also known as simplex-constrained regression.

\textbf{Experimental Details}: We look at a convex hull sampled from the unit hypercube $[0, 1]^d$ for $d \in [10, 15, \ldots, 50]$. For each hypercube, we sample 50 points uniformly on each of its surfaces, giving a convex hull $X$ with $n=100d$ data points. 

Once $X$ is sampled, 50 $y$'s are created outside the hypercube perpendicular to a surface and unit length away from it. This is done by considering the 50 points in $X$ lying on a randomly selected surface of the hypercube. A point $y_\text{true}$ is created as a random convex combination of these points. The point $y$ can then be created perpendicular to this surface and a unit length away from $y_\text{true}$, and thus also from the convex hull of $X$.

Each $y$ is then projected onto $X$ using CS, EGD, and PFW. These algorithms are ran until a solution, $\hat{y}$, is found such that $\|\hat{y} - y_\text{true}\|\leq 1e^{-5}$ or $10\,000$ iterations have been made. We do not implement PGD and Frank-Wolfe due to their inefficiency in practice. 

\textbf{Implementation Details}: 

The learning rate for EGD, PFW, and CS is found through a line search. In the case of the PFW and CS algorithms, an explicit solution can be found and used. At the same time, EGD implements a back-tracking linear search with Armijo conditions \citep{numerical_optimization_book} to find an approximate solution.

Experiments were written in Python and ran on Google Colab. The code can be found on GitHub\footnote{https://github.com/infamoussoap/ConvexHull}. The random data is seeded for reproducibility.

\textbf{Results}: The results can be seen in Fig. \ref{fig:results}. For $d=10$, we see that PFW outperforms both CS and EGD in terms of the number of iterations required and the time taken. But for $d>10$, on average, CS converges with the least iterations and time taken.
 
\begin{figure}[tbhp]
\centering
\subfloat[Average steps taken in log-scale]{\label{fig:a}\includegraphics[width=0.48\textwidth]{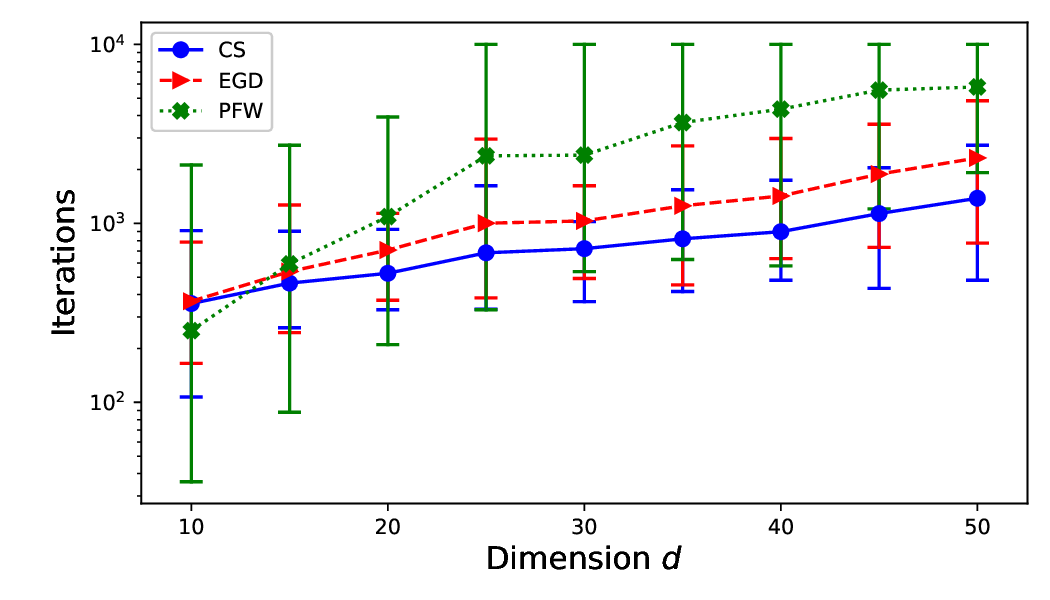}}
\subfloat[Average time taken in log-scale]{\label{fig:b}\includegraphics[width=0.48\textwidth]{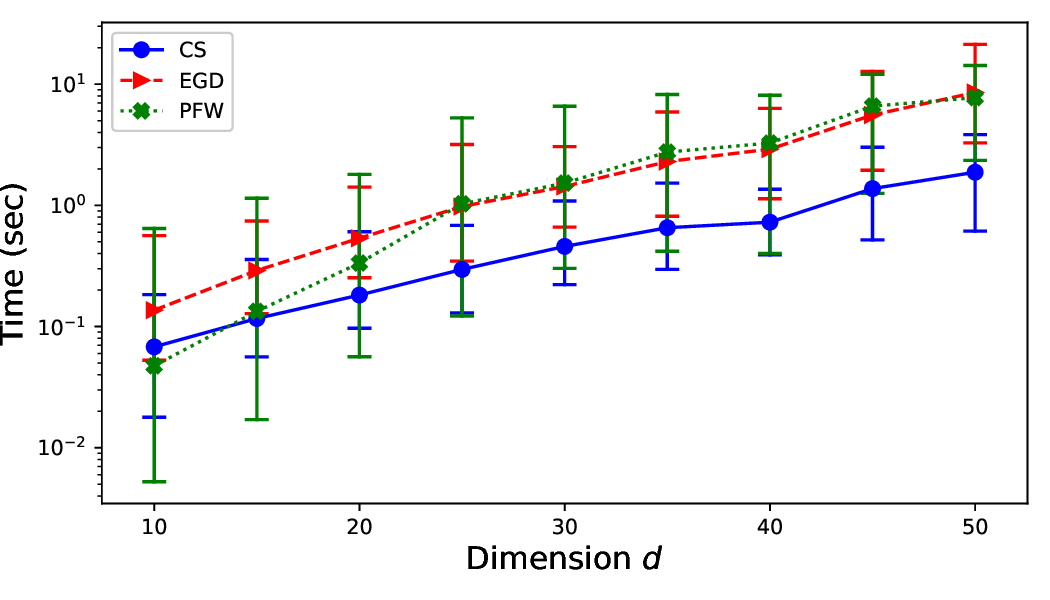}}
\caption{Number of steps and time required for PFW, EGD, and CS to project 50 randomly sampled points onto the $d$-hypercube. The bars indicate the minimum and maximum values.}
\label{fig:results}
\end{figure}

\subsection{Optimal Question Weighting}\label{section:optimal_question_weighting}
It is often desirable that the distribution of exam marks matches a target distribution, but this rarely happens. Modern standardized tests (e.g. IQ exams) solve this problem by transforming the distribution of the raw score of a given age group so it fits a normal distribution \citep{iq_Bartholomew, iq_Gottfredson, iq_makintosh}.

While IQ exams have many criticisms \citep{iq_test_bad_3, iq_test_bad_1, iq_test_bad_2}, we are interested in the raw score. As noted by \cite{iq_Gottfredson}, the raw score has no intrinsic meaning as it can be boosted by adding easier questions to the test. We also argue it is hard to predict the difficulty of a question relative to an age group and, thus, even harder to give it the correct weight. Hence making the raw score a bad reflection of a person's performance.

Here we propose a framework to find an optimum weighting of questions such that the weighted scores will fit a target distribution. A demonstration can be seen in Fig. \ref{fig:mark_distribution}.

Consider $d$ students taking an exam with $n$ true or false questions. For simplicity, assume that person $j$ getting question $i$ correct can be modeled as a random variable $\mathcal{X}_{i, j}$ for $i=1,\ldots, n$ and $j=1,\ldots, d$. Consider the discrete distribution of $X_j = \sum_i w_i{\mathcal{X}}_{i,j}$ for some $w\in\Delta^n$, the weighted mark of person $j$. This distribution can be approximated as a continuous distribution,
\begin{equation*}
    \rho_\varepsilon(z)\ =\ \frac{1}{d}\sum_{j=1}^d \mu_\varepsilon(z - X_j)\quad \text{with}\quad \mu_\varepsilon(x)\ =\ \frac{1}{\varepsilon}\mu(x/\varepsilon),
\end{equation*}
$\varepsilon>0$ and $\mu$ is a continuous probability distribution, i.e. $\mu \geq 0$ and $\int \mu(x) dx = 1$. This is also known as kernel density estimation (KDE).

We want to minimize the distance between $\rho_\varepsilon$ and some target distribution $f$. A natural choice is the relative entropy,
\begin{equation*}
    \min_{w\in \Delta^n} D(\rho_\varepsilon | f)\ =\ \min_{w\in \Delta^n} \int \rho_\varepsilon(x)\log\left(\frac{\rho_\varepsilon(x)}{f(x)}\right)\, dx,
\end{equation*}
of which we take its Riemann approximation,
\begin{align}\label{eq:mark_distribution_discrete}
    \min_{w\in \Delta^n} \hat{D}(\rho_\varepsilon | f)\ =\ \min_{w\in \Delta^n} \sum_{k=1}^M\rho_\varepsilon(x_k)\log\left(\frac{\rho_\varepsilon(x_k)}{f(x_k)}\right) (x_k - x_{k-1}),
\end{align}
where $\{x_k\}_{0\leq k\leq M}$ is a partition of a finite interval $[a,b]$. 

We remark that this problem is not convex, as $\mu$ cannot be chosen to be convex w.r.t. $w$ and be a probability distribution.

This is similar to robust location parameter estimation \citep{robust_location_parameter, robust_statistics_textbook}, which considers data sampled from a contaminated distribution (\textit{i.e.} an uncertain mixture of two known distributions), for instance, when a predicting apparatus failures from two populations. Specifically for some contamination level $0 < \alpha \ll 1$, data was sampled at a rate of $1-\alpha$ from a distribution, $G$ (\textit{i.e.} easy questions), and $\alpha$ from an alternate distribution, $H$ (\textit{i.e.} hard questions). Thus, data is sampled from the distribution
\begin{align*}
    F\ =\ (1-\alpha)G + \alpha H.
\end{align*}

Under this contaminated distribution, robust parameter estimation finds point estimates (e.g., mean and variance) that are robust when $\alpha$ is small. In this sense, robust parameter estimation is interested in finding weights, $w$, that make the KDE, $\rho_{\varepsilon}(z)$, robust to contamination levels $\alpha$. Instead, we wish to make the KDE match a target distribution.

\begin{figure}[tbhp]
\centering
\includegraphics[width=0.98\textwidth]{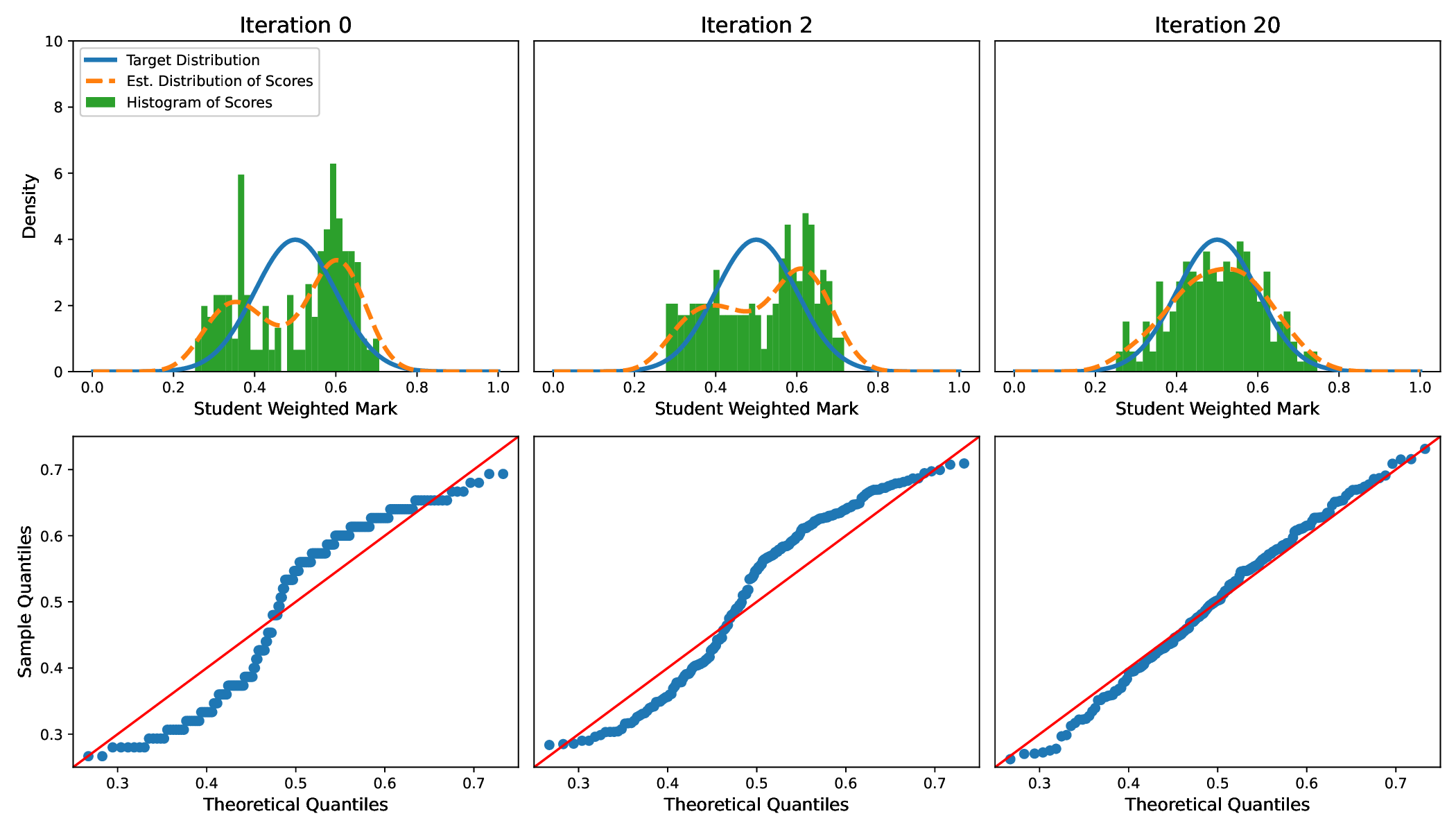}
\caption[LoF entry]{Optimal question weighting for (randomly generated) exam scores with 200 students and 75 questions. The setup follows the experimental details in Section \ref{section:optimal_question_weighting}. The kernel density estimate uses a truncated unit normal distribution with $\varepsilon=0.05$, and the target distribution is a truncated normal distribution with a mean of 0.5 and a standard deviation of 0.1. We take $w^0=(0.01,\ldots, 0.01)$, and the Cauchy-Simplex is applied, with each step using a backtracking line search. The resulting weighted histogram and kernel density estimate is shown.

The distribution of the weighted marks is shown on the top row, and its QQ plots against a normal distribution of mean 0.5 and a standard deviation of 0.1 are shown on the bottom row.

At iterations 0, 5, and 20, the weighted scores have a mean of 0.499, 0.514, and 0.501, with standard deviations of 0.128, 0.124, and 0.109, respectively.}
\label{fig:mark_distribution}
\end{figure}

\textbf{Experiment Details}: We consider 25 randomly generated exam marks, each having $d=200$ students taking an exam with $n=75$ true or false questions. For simplicity, we assume that $\mathcal{X}_{i,j}\sim \text{Bernoulli}(q_i s_j)$ where $0<q_i<1$ is the difficulty of question $i$ and $0<s_j<1$ the $j$-th student's smartness.

For each scenario, $q_i = 7/8$ for $1\leq i\leq 60$ and $q_i=1/5$ for $60 < i \leq 75$, while $s_j = 7/10$ for $1\leq j\leq 120$ and $s_j=1/2$ for $120 < j \leq 200$. ${\mathcal{X}}_{i,j}\sim\text{Bernoulli}(q_is_j)$ are then sampled. This setup results in a bimodal distribution, with an expected average of $0.532$ and an expected standard deviation of $0.1206$, as shown in Figure \ref{fig:mark_distribution}.

For the kernel density estimate, $\mu(x)$ is chosen as a unit normal distribution truncated to $[0, 1]$, with smoothing parameter $\varepsilon=0.05$. Similarly, $f$ is a normal distribution with mean $0.5$ and variance $0.1$, truncated to $[0, 1]$. We take the partition $\{k/400\}_{0\leq k\leq 400}$ for the Riemann approximation (\ref{eq:mark_distribution_discrete}).

The algorithms CS, EGD, and PFW are ran for 150 iterations.

\textbf{Implementation Details}: The learning rate for EGD, PFW, and CS is found through a line search. However, explicit solutions are not used. Instead, a back-tracking line search with Armijo conditions is used to find an approximate solution.

Experiments were written in Python and ran on Google Colab and can be found on GitHub\footnote{https://github.com/infamoussoap/ConvexHull}. The random data is seeded for reproducibility.

\textbf{Results}: A table with the results can be seen in Table \ref{tab:optimal_sample_weighting}. In summary, of the 25 scenarios, PFW always produces solutions with the smallest relative entropy, with CS producing the largest relative entropy 13 times and EGD 12 times. For the time taken to make the 150 steps, PFW is the quickest 15 times, EGD 7 times, and CS 3 times. At the same time, EGD is the slowest 13 times, CS 7 times, and PFW 5 times.

It is perhaps expected that PFW outperforms both EGD and CS due to the low dimensionality of this problem. However, the CS produces similar relative entropies to EGD while maintaining a lower average run time of 5.22 seconds compared to the average run time of 5.68 sec for EGD. 

\begin{table}[]
\caption{Results for optimal question weighting after running CS, EGD, and PFW for 150 iterations. The final distance (relative entropy) and the time taken in seconds are shown. Bold represents the minimum (best) value, and underline represents the maximum (worst) value.}
\label{tab:optimal_sample_weighting}
\centering
\begin{tabular}{l|cc|cc|cc}
\toprule
 & \multicolumn{2}{c|}{CS}&\multicolumn{2}{c|}{EGD} &\multicolumn{2}{c}{PFW} \\
\midrule
Trial & Distance & Time & Distance & Time & Distance & Time \\
\midrule
1 & \underline{0.032432} & 8.14 & 0.032426 & \underline{10.98} & \textbf{0.032114} & \textbf{6.02} \\
2 & 0.010349 & \underline{5.95} & \underline{0.010535} & 5.08 & \textbf{0.010101} & \textbf{4.46} \\
3 & 0.016186 & \underline{5.74} & \underline{0.016252} & \textbf{4.93} & \textbf{0.015848} & 5.50 \\
4 & 0.025684 & 4.62 & \underline{0.025726} & \underline{6.19} & \textbf{0.025309} & \textbf{4.51} \\
5 & \underline{0.020561} & 4.63 & 0.020486 & \underline{6.03} & \textbf{0.020213} & \textbf{4.50} \\
6 & \underline{0.016559} & \underline{5.58} & 0.016514 & 5.00 & \textbf{0.016287} & \textbf{4.52} \\
7 & \underline{0.025957} & \underline{5.42} & 0.025867 & \textbf{4.87} & \textbf{0.025757} & 5.38 \\
8 & \underline{0.014506} & 4.77 & 0.014343 & \underline{6.14} & \textbf{0.013504} & \textbf{4.28} \\
9 & 0.032221 & 4.68 & \underline{0.032412} & \underline{6.02} & \textbf{0.032028} & \textbf{4.55} \\
10 & 0.023523 & \underline{5.59} & \underline{0.023528} & \textbf{4.92} & \textbf{0.023232} & 5.34 \\
11 & 0.016153 & \textbf{4.93} & \underline{0.016231} & 5.01 & \textbf{0.015792} & \underline{5.63} \\
12 & 0.035734 & 4.53 & \underline{0.035738} & \underline{6.07} & \textbf{0.035212} & \textbf{4.22} \\
13 & 0.030205 & 4.53 & \underline{0.030234} & \underline{6.13} & \textbf{0.029859} & \textbf{4.37} \\
14 & \underline{0.021725} & 5.80 & 0.021598 & \textbf{4.99} & \textbf{0.021282} & \underline{5.85} \\
15 & \underline{0.030026} & \textbf{4.44} & 0.029982 & 4.99 & \textbf{0.029751} & \underline{5.64} \\
16 & \underline{0.009212} & 4.75 & 0.009182 & \underline{5.94} & \textbf{0.008931} & \textbf{4.27} \\
17 & 0.015573 & 5.22 & \underline{0.015661} & \underline{5.46} & \textbf{0.015188} & \textbf{4.48} \\
18 & \underline{0.017681} & \underline{5.69} & 0.017618 & \textbf{4.92} & \textbf{0.017321} & 5.48 \\
19 & \underline{0.017888} & \textbf{4.64} & 0.017874 & \underline{5.34} & \textbf{0.017283} & 5.11 \\
20 & 0.013597 & 4.55 & \underline{0.013719} & \underline{6.04} & \textbf{0.013075} & \textbf{4.42} \\
21 & \underline{0.016933} & \underline{5.76} & 0.016780 & 4.84 & \textbf{0.016687} & \textbf{4.77} \\
22 & \underline{0.032185} & 5.80 & 0.032141 & \textbf{5.03} & \textbf{0.032039} & \underline{6.03} \\
23 & \underline{0.018377} & 4.69 & 0.018250 & \underline{6.06} & \textbf{0.018084} & \textbf{4.54} \\
24 & 0.031167 & 4.59 & \underline{0.031211} & \underline{6.10} & \textbf{0.030820} & \textbf{4.55} \\
25 & 0.035608 & 5.46 & \underline{0.035674} & \textbf{4.98} & \textbf{0.035408} & \underline{5.98} \\
\midrule
Average & & 5.22 & & 5.68 & & 4.97\\
\bottomrule
\end{tabular}
\end{table}

\subsection{Prediction from Expert Advice}
{Consider $N$ `experts' (e.g., Twitter) who give daily advice for $1\leq t \leq T$ days. By taking advice from expert $i$ on day $t$, we incur a loss $l^t_i\in[0,1]$. However, this loss is not known beforehand, only once the advice has been taken. Since the wisdom of the crowd is typically better than any given expert \citep{Landemore_Elster_2012}, we'd like to take the weighted average of our experts, giving a daily loss of $w^t\cdot l^t$, where $w^t\in \Delta^N$ are the weights given to the experts. This problem is also known as the multi-armed bandit problem.}

A simple goal is to generate a sequence of weight vectors $\{w^t\}_t$ to minimize the averaged 
expected loss. This goal is, however, a bit too ambitious as the loss vectors $l^t$ are not known beforehand. An easier problem is to find a sequence, $\{w^t\}_t$, such that its averaged expected loss approaches the average loss of the best expert as $T\to\infty$, that is
\begin{align*}
    \frac{1}{T}R_T \to 0, \quad \text{where} \quad R_T\ =\ \sum_{t=1}^Tw^t\cdot l^t - \min_i \sum_{t=1}^T l^t_i 
\end{align*}
as $T\to\infty$. $R_T$ is commonly known as the regret of the strategy $\{w^t\}_t$.

Previous works \citep{weighted_majority_algorithm, using_expert_advice, adaboosting, survey_multiplicative_weights} all yield $\mathcal{O}(\sqrt{T\log N})$ convergence rate for the regret.

\begin{theorem}\label{thm:expert_advice}
    Consider a sequence of adversary loss vectors $l^t \in [0, 1]^N$. For any $u\in\Delta^N$, the regret generated by the Cauchy-Simplex scheme
    \begin{align*}
        w^{t+1}\ =\ w^t(1 - \eta^t \Pi_{w^t}\nabla f^t)\quad\text{where}\quad \nabla f^t\ =\ l^t,
    \end{align*}
    is bounded by
    \begin{align*}
        \sum_{t=1}^T w^t \cdot l^t - \sum_{t=1}^T u \cdot l^t \ \leq\ \frac{D(u|w^1)}{\eta} + \frac{T\eta}{2(1-\eta)},
    \end{align*}
    for a fixed learning rate $\eta^t = \eta < 1$.
    
    In particular, taking $w^1=(1/N,\ldots, 1/N)$ and $\eta = \frac{\sqrt{2 \log (N)}}{\sqrt{2\log (N)}+\sqrt{T}}$ gives the bound
    \begin{align*}
        \sum_{t=1}^T w^t\cdot l^t - \sum_{t=1}^T u\cdot l^t \ \leq\  \sqrt{2 T \log (N)} + \log (N).
    \end{align*}
    Moreover, this holds when $u=e_j$, where $j$ is the best expert and $e_j$ is the standard basis vector.
\end{theorem}
The proof can be found in Appendix \ref{sec:appendix_expert_advice_proof}.

\subsection{Universal Portfolio}
Consider an investor with a fixed-time trading horizon, $T$, managing a portfolio of $N$ assets. Define the price relative for the $i$-th stock at time $t$ as $x^t_i = C^t_i / C^{t-1}_i$, where $C^t_i$ is the closing price at time $t$ for the $i$-th stock. So today's closing price of asset $i$ equals $x_i^t$ times yesterday's closing price, \textit{i.e.} today's price relative to yesterday's.

A portfolio at day $t$ can be described as $w^t \in \Delta^N$, where $w_i^t$ is the proportion of an investor's total wealth in asset $i$ at the beginning of the trading day. Then the wealth of the portfolio at the beginning of day $t+1$ is $w^t\cdot x^t$ times the wealth of the portfolio at day $t$.

Consider the average log-return of the portfolio
\begin{align*}
    \frac{1}{T}\log\left(\prod_{t=1}^T w^t \cdot x^t\right)\ =\ \frac{1}{T}\sum_{t=1}^T\log(w^t\cdot x^t).
\end{align*}
Similarly to predicting with expert advice, it is too ambitious to find a sequence of portfolio vectors $\{w^t\}_t$ that maximizes the average log-return. Instead, we wish to find such a sequence that approaches the best fixed-weight portfolio, \textit{i.e.}
\begin{align*}
    \frac{1}{T} LR_T\to 0, \quad \text{where}\quad LR_T\ =\ \sum_{t=1}^T\log(u\cdot x^t) - \sum_{t=1}^T\log(w^t\cdot x^t)
\end{align*}
as $T\to\infty$, for some $u\in\Delta^N$. If such a sequence can be found, $\{w^t\}_t$ is a universal portfolio. $LR_T$ is commonly known as the log-regret. 

Two standard assumptions are made when proving universal portfolios: (1) For every day, all assets have a bounded price relative, and at least one is non-zero, \textit{i.e.} $0 < \max_{i} x^t_i < \infty$ for all $t$, and (2) No stock goes bankrupt during the trading period, \textit{i.e.} $a :=\min_{i,t} x_i^t > 0$, where $a$ known as the market variability parameter. This is also known as the no-junk-bond assumption.

Over the years, various bounds on the log-regret have been proven under both assumptions. Some examples include \cite{up_cover}, in his seminal paper, with $\mathcal{O}(\log T)$, \cite{up_helmbold} with $\mathcal{O}(\sqrt{T \log N})$, \cite{up_newton} with $\mathcal{O}(N^{1.5}\log(NT))$, \cite{hazan_kale} with $\mathcal{O}(N\log(T + N))$, and \cite{gaivoronski_stella_1} with $\mathcal{O}(C^2\log(T))$ where $C = \sup_{x\in\Delta^N}\|\nabla_b \log(b\cdot x)\|$ and adding an extra assumption on independent price relatives. Each with varying levels of computational complexity.

\begin{remark}
    Let $x^t\in[a,b]^N$ be a bounded sequence of price relative vectors for $0\leq a\leq b < \infty$. Since the log-regret is invariant under re-scalings of $x^t$, w.l.o.g. we can look at the log-regret for the re-scaled return vectors $x^t\in[\tilde{a}, 1]^N$ for $0\leq \tilde{a} \leq 1$.
\end{remark}

\begin{theorem}\label{thm:portfolio_cauchy_simplex}
Consider a bounded sequence of price relative vectors $x^t\in [a, 1]^N$ for some positive constant $0<a\leq1$ (no-junk-bond), and $\max_i x_i^t=1$ for all $t$. Then the log-regret generated by the Cauchy-Simplex 
\begin{align*}
    w^{t+1}\ =\ w^t(1 - \eta \Pi_{w^t} \nabla f^t),\quad\text{where}\quad \nabla f^t\ =\ - \frac{x^t}{w^t\cdot x^t},
\end{align*}
is bounded by
\begin{align*}
    \sum_{t=1}^T\log(u\cdot x^t) - \sum_{t=1}^T\log(w^t \cdot x^t)\ \leq\ \frac{D(u|w^1)}{\eta} + \frac{T\eta}{2a^2(1-\eta)},
\end{align*}
for any $u\in \Delta^N$ and $0 < \eta \leq 1$.

In particular, taking $w^1=(1/N,\ldots, 1/N)$ and $\eta = \frac{a\sqrt{2 \log(N)}}{a\sqrt{2 \log(N)} + \sqrt{T}}$ gives the bound
\begin{align*}
    \sum_{t=1}^T\log(u\cdot l^t) - \sum_{t=1}^T \log(w^t\cdot l^t)\ \leq\ \frac{\sqrt{2 T \log (N)}}{a}+\log (N).
\end{align*}
\end{theorem}
The proof can be found in Appendix \ref{sec:appendix_universal_portfolio_proof}.

\textbf{Experimental Details}: We look at the performance of our algorithm on four standard datasets used to study the performance of universal portfolios: (1) {NYSE} is a collection of 36 stocks traded on the New York Stock Exchange from July 3, 1962, to Dec 31, 1984, (2) {DJIA} is a collection of 30 stocks tracked by the Dow Jones Industrial Average from Jan 14, 2009, to Jan 14, 2003, (3) {SP500} is a collection of the 25 largest market cap stocks tracked by the Standard \& Poor's 500 Index from Jan 2, 1988, to Jan 31, 2003, and (4) TSE is a collection of 88 stocks traded on the Toronto Stock Exchange from Jan 4, 1994, to Dec 31, 1998.\footnote{The datasets were original found on \url{http://www.cs.technion.ac.il/~rani/portfolios/}, but is now unavailable. It was retrieved using the WayBack Machine \url{https://web.archive.org/web/20220111131743/http://www.cs.technion.ac.il/~rani/portfolios/}.}

Two other portfolio strategies are considered: (1) Helmbold \textit{et. al.} \cite{up_helmbold} (EGD), who uses the EGD scheme $w^{t+1} = w^t\exp(\eta \nabla f^t) / \sum_i w^t_i\exp(\eta \nabla_i f^t)$, with $\nabla f^t = x^t/w^t\cdot x^t$, and (2) Buy and Hold (B\&H) strategy, where one starts with an equally weighted portfolio and the portfolio is left to its own devices.

Two metrics are used to evaluate the performance of the portfolio strategies: (1) The Annualized Percentage Yield: $\text{APY} = \text{R}^{1/y}-1$, where $R$ is the total return over the full trading period, and $y=T/252$, where 252 is the average number of annual trading days, and (2) The Sharpe Ratio: $\text{SR} = (\text{APY} - R_f)/\sigma$, where $\sigma^2$ is the variance of the daily returns of the portfolio, and $R_f$ is the risk-free interest rate. Intuitively, the Sharpe ratio measures the performance of the portfolio relative to a risk-free investment while also factoring in the portfolio's volatility. Following \cite{pamr}, we take $R_f=4\%$.

We take the learning rate as the one used to prove that CS and EGD are universal portfolios. In particular, $\eta^{\text{CS}} = \frac{a\sqrt{2\log N}}{a\sqrt{2\log N} + \sqrt{T}}$ and $\eta^{\text{EGD}} = 2a\sqrt{2\log(N) / T}$, respectively, where $a$ is the market variability parameter. We assume that the market variability parameter is given for each dataset.

Experiments were written in Python and can be found on GitHub\footnote{\url{https://github.com/infamoussoap/UniversalPortfolio}}.

\textbf{Results}: A table with the results can be seen in Table \ref{tab:universal_portfolio}. For the NYSE, DJIA, and SP500 datasets, CS slightly outperforms EGD in both the APY and Sharpe ratio, with EGD having a slight edge on the APY for the NYSE dataset. But curiously, the B\&H strategy outperforms both CS and EGD on the TSE.

We remark that this experiment does not reflect real-world performance, as the market variability parameter is assumed to be known, transaction costs are not factored into our analysis, and the no-junk-bond assumption tends to overestimate performance \citep{survivorship_bias_1, survivorship_bias_2, survivorship_bias_3}. However, this is outside of the scope of this paper. It is only shown as a proof of concept.

\begin{table}[]
\caption{Performance of different portfolio strategies on different datasets. Bold represents the maximum (best) value, and underline represents the minimum (worst) value.}
\label{tab:universal_portfolio}
\centering
\begin{tabular}{l|cc|cc|cc}
\toprule
 & \multicolumn{2}{c|}{CS}&\multicolumn{2}{c|}{EGD} &\multicolumn{2}{c}{B\&H} \\
\midrule
Dataset & APY & Sharpe & APY & Sharpe & APY & Sharpe\\
\midrule
NYSE & 0.162 & \textbf{14.360} & \textbf{0.162} & 14.310& \underline{0.129} & \underline{9.529} \\
DJIA & \textbf{-0.099} & \textbf{-8.714} & -0.101 & -8.848& \underline{-0.126} & \underline{-10.812} \\
SP500 & \textbf{0.104} & \textbf{4.595} & 0.101 & 4.395& \underline{0.061} & \underline{1.347} \\
TSE & 0.124 & 10.225& \underline{0.123} & \underline{10.204} & \textbf{0.127} & \textbf{10.629} \\
\bottomrule
\end{tabular}
\end{table}

\section{Conclusion}
This paper presents a new iterative algorithm, the Cauchy-Simplex, to solve convex problems over a probability simplex. Within this algorithm, we find ideas from previous works which only capture a portion of the simplex constraint. Combining these ideas, the Cauchy-Simplex provides a numerically efficient framework with nice theoretical properties.

The Cauchy-Simplex maintains the linear form of Projected Gradient Descent, allowing one to find analytical solutions to a line search for certain convex problems. But unlike projected gradient descent, this analytical solution will remain in the probability simplex. A backtracking line search can be used when an analytical solution cannot be found. However, this requires an extra parameter, the maximum candidate step size. The Cauchy-Simplex provides a natural answer as a maximum learning rate is required to enforce positivity, rather than the exponentials used in Exponentiated Gradient Descent. 

Since the Cauchy-Simplex satisfies both constraints of the probability simplex in its iteration scheme, its gradient flow can be used to prove convergence for differentiable and convex functions. This implies the convergence of its discrete linear scheme. This is in contrast to EGD, PFW, and PGD, in which its discrete nature is crucial in satisfying both constraints of the probability simplex. More surprisingly, we find that in the proofs, formulas natural to probability, \textit{i.e.}, variance, and relative entropy, are necessary when proving convergence.

We believe that the strong numerical results and simplicity seen through its motivating derivation, gradient flow, and iteration scheme make it a strong choice for solving problems with a probability simplex constraint.

\acks{We thank Prof. Johannes Ruf for the helpful discussion and his suggestion for potential applications in the multi-armed bandit problem, which ultimately helped the proof for universal portfolios. We also thank the anonymous reviewers for their helpful suggestions about improving the presentation and readability of this paper. The authors declare no competing interests.}


\newpage


\appendix
\section{Gradient Flow for Orthogonal Matrix Constraint}\label{sec:cayley_transform_derivation}
Consider the optimization problem
\begin{align*}
    \min_{Q\in \mathcal{S}^n} f(Q),\quad\text{where}\quad \mathcal{S}^n=\{Q\in\mathbb{R}^{n\times n}\, |\, QQ^T=I\}.
\end{align*}
Orthogonal matrices with $\det(Q)=1$ can be parameterized using the Cayley Transform $Q = 2(I - X)^{-1} - I$, where  $X\in\mathbb{R}^{n\times n}$ is an anti-symmetric matrix, and let $F(X)= f(Q)$. The anti-symmetric matrix $X$ can be parameterized as
\begin{align*}
    X\ =\ \sum_{j > k} X_{jk}\, E^{jk},
\end{align*}
for $X_{jk}\in\mathbb{R}$ and $E^{jk}_{pq}=\delta_{pj}\delta_{qk}-\delta_{pk}\delta_{qj}$.

Consider the gradient flow
\begin{align*}
    \frac{dQ_{pq}}{dt}\ &=\ -\sum_{jk}\frac{\partial Q_{pq}}{\partial X_{jk}}\frac{d X_{jk}}{dt}\\
    &=\  -\sum_{jk}\frac{\partial Q_{pq}}{\partial X_{jk}}\sum_{ab} \frac{\partial f}{\partial Q_{ab}}\, \frac{\partial Q_{ab}}{\partial X_{jk}}.
\end{align*}
Under the Cayley Transform,
\begin{align*}
    \frac{\partial Q}{\partial X_{jk}}\ &=\ 2(I-X)^{-1} \frac{\partial X}{\partial X_{jk}}(I-X)^{-1}\\
    &=\ 2\Omega\, E^{jk}\, \Omega.
\end{align*}
Thus
\begin{align*}
    \frac{\partial Q_{ab}}{\partial X_{jk}}\ &=\ 2\sum_{lm}\Omega_{al}\,E^{jk}_{lm}\,\Omega_{mb}\\
    &=\ 2(\Omega_{aj}\Omega_{kb}-\Omega_{ak}\Omega_{jb})
\end{align*}

The gradient flow for the latent variables $X_{jk}$ is therefore
\begin{align*}
    \frac{dX_{jk}}{dt}\ &=\ 2\sum_{ab} \frac{\partial f}{\partial Q_{ab}}(\Omega_{aj}\Omega_{kb}-\Omega_{ak}\Omega_{jb})\\
    &=\ 2\left(\Omega^T \frac{\partial f}{\partial Q_{ab}} \Omega^T- \Omega \left(\frac{\partial f}{\partial Q_{ab}}\right)^T \Omega\right)\\
    &=\ 2(\Lambda\, \Omega^T - \Omega \,\Lambda^T)_{jk},
\end{align*}
where $\Lambda = \Omega^T \frac{\partial f}{\partial Q}$. Converting this to the gradient flow for the orthogonal matrix $Q$ yields 
\begin{align*}
    \frac{dQ_{pq}}{dt}\ &=\ -2\sum_{jk}(\Omega_{pj}\Omega_{kq}-\Omega_{pk}\Omega_{jq})(\Lambda \Omega^T - \Omega \Lambda^T)_{jk}\\
    &=\ -2[\Omega(\Lambda \Omega^T - \Omega \Lambda^T)\Omega - \Omega(\Lambda \Omega^T - \Omega \Lambda^T)^T\Omega]\\
    &=\ -4\Omega(\Lambda \Omega^T - \Omega \Lambda^T)\Omega,
\end{align*}
as required.

\section{Convergence Proofs}
We will use the notation $\Pi^t=\Pi_{w_t}$ and $\nabla f^t = \nabla_w f(w^t)$ for the remaining section.
\subsection{Decreasing Relative Entropy}
\begin{theorem}\label{thm:appendix_relative_entropy_diff}
    For any $u\in\Delta^N$, each iteration of the linear Cauchy-Simplex (\ref{eq:discrete_time_scheme}) satisfies the bound
    \begin{align*}
        D(u | w^{t+1}) - D(u|w^t)\ \leq\ \eta^t u\cdot(\Pi^t\nabla f^t) + C_{\gamma_t}(\eta^t)^2 u\cdot(\Pi^t\nabla f^t)^2,
    \end{align*}
    where $C_{\gamma_t} = \gamma_t^{-2}\log(e^{-\gamma_t}/(1 - \gamma_t))$ and $\eta^t = \gamma_t \eta^{t,\max}$ with $\gamma_t\in(0,1)$.
\end{theorem}

\begin{proof}
    By the CS update scheme (\ref{eq:discrete_time_scheme})
    \begin{align*}
        D(u|w^{t+1}) - D(u|w^{t})\ &=\ \sum_i u_i\log\left(w^{t+1}_i/w^t_i\right)\\
        &=\ \sum_iu_i\log\left(\frac{1}{1-\eta^t\Pi^t\grad_i f^t}\right)\\
        &=\ \sum_iu_i\log\left(\frac{e^{-\eta^t\Pi^t\grad_i f^t}}{1-\eta^t\Pi^t\grad_i f^t}e^{\eta^t\Pi^t\grad_i f^t}\right)\\
        &=\ \eta^t u\cdot(\Pi^t\nabla f^t) + \sum_i u_i\log\left(\frac{e^{-\eta^t\Pi^t\grad_i f^t}}{1-\eta^t\Pi^t\grad_i f^t}\right).
    \end{align*}

    Since the learning rate can be written as $\eta^t = \gamma_t\eta^{t,\max}$, with $\gamma_t\in(0,1)$,
    \begin{align*}
        \eta^t \Pi^t\nabla_i f^t\ \leq\ (\gamma_t \eta^{t,\max})\max_i \Pi^t \nabla_i f^t\ =\ \gamma_t\ <\ 1.
    \end{align*}

    {It can be shown that $x^{-2}\log(e^{-x}/(1-x)) > 0$ for $x \in (0, 1)$ and is an increasing function on $(0, 1)$ with $x^{-2}\log(e^{-x}/(1-x))\to\infty$ as $x\to 1$. Thus, in the interval $x \in (0, \gamma_t]$ with $\gamma_t < 1$, this can be upper-bounded by
    \begin{align*}
        0\ <\ x^{-2}\log[e^{-x}/(1-x)]\ \leq\ \gamma_t^{-2}\log[e^{-\gamma_t}/(1-\gamma_t)]\ =\ C_{\gamma_t},\qquad\text{for}\qquad x\in (0,\gamma_t].
    \end{align*}
    This yields the inequality $0< \log(e^{-x}/(1-x))\leq C_{\gamma_t}\, x^2$ for $0 < x \leq \gamma_t$.}
    
    Since $C_{\gamma_t} > 0$ for $\gamma_t \in (0, 1)$ and $C_{\gamma_t} \to \infty$ as $\gamma_t \to 1$, we have the bound
    \begin{align*}
        0\ \leq\ \log\left(\frac{e^{-\eta^t\Pi^t\grad_i f^t}}{1-\eta^t\Pi^t\grad_i f^t}\right)\ \leq\ C_{\gamma_t}(\eta^t)^2(\Pi^t \nabla_i f^t)^2. 
    \end{align*}
    Giving the required inequality
    \begin{align*}
        D(u|w^{t+1}) - D(u|w^{t})\ \leq\ \eta^t u\cdot(\Pi^t\nabla f^t) + C_{\gamma_t}(\eta^t)^2 u\cdot(\Pi^t \nabla f^t)^2.
    \end{align*}
\end{proof}

\subsection{Progress Bound}
\begin{theorem}\label{thm:discrete_decreasing}
Let $f$ be convex, differentiable, and $\nabla f$ is $L$-Lipschitz continuous. Then each iteration of the linear Cauchy-Simplex (\ref{eq:discrete_time_scheme}) guarantees
\begin{equation*}
f(w^{t+1})\ \leq\ f(w^t) - \frac{\eta^t}{2}\var [\nabla f^t \, | \,  w^t],\quad\text{for}\quad 0 < \eta^t < \min\bigg{\{}\frac{1}{L}, \eta^{t,\max}\bigg{\}},
\end{equation*}
where $\eta^{t,\max}$ is defined in (\ref{eq:discrete_time_scheme}).
\end{theorem}
\begin{proof}
    Since $f$ is convex with $\nabla f$ Lipschitz continuous, 
    {the descent lemma yields \citep{Bertsekas_2016}}
    \begin{equation}\label{eq:theorem_14_descent_lemma}
        f(w^{t+1})\ \leq\ f(w^t) + \nabla f^t \cdot(w^{t+1} - w^t) + \frac{L}{2}\|w^{t+1} - w^t\|^2.
    \end{equation}
    Our iteration scheme gives that
    \begin{equation*}
        \|w^{t + 1} - w^t\|^2\ =\ (\eta^t)^2\sum_i\Big(w^t_i \, \Pi^t \grad_i f^t\Big)^2\ \leq\ (\eta^t)^2\sum_i w_i^t(\Pi^t \grad_i f^t)^2,
    \end{equation*}
    since $0\leq w_i^t\leq 1$. Hence
    \begin{equation*}
        f(w^{t+1})\ \leq\ f(w^t) - \frac{\Var[\nabla f^t|w_t]}{2L} \left(  2 z - z^2 \right),
    \end{equation*}
    where $z=\eta^{t} L$. However, $- ( 2z - z^{2} ) \le -z$ for $0 \le z \le 1$. Therefore,
    \begin{equation*}
        f(w^{t+1})\ \leq\ f(w^t) - \frac{\eta^t}{2} \var [\nabla f^t \, | \,  \, w^t],\quad\text{for}\quad 0< \eta^t \leq \min\bigg{\{}\frac{1}{L}, \eta^{t,\max}\bigg{\}}.
    \end{equation*}
\end{proof}

\subsection{Proof of Lemma \ref{lemma:asymptotic_convergence}}\label{sec:asymptotic_convergence}
\begin{proof}
    Notice that for $\eta^t\in (0, \eta^{t,\max}-\varepsilon]$, the bound (\ref{eq:theorem_14_descent_lemma}) still holds, and that a line search minimizes the left-hand side of (\ref{eq:theorem_14_descent_lemma}). Bounding the right-hand side yields
    \begin{align*}
        f(w^{t+1})\ \leq\ f(w^t) - \frac{\eta^t}{2}\var[\nabla f^t\,|\,w^t],\quad\text{for}\qquad \eta^t = \min\bigg{\{}\frac{1}{L}, \eta^{t,\max} - \varepsilon\bigg{\}}.
    \end{align*}
    
    As shown in Theorem \ref{thm:decreasing}, $\var[\nabla f^t|w^t]=0$ only when $w^t=w^*$. Thus w.l.o.g., for finite $T>0$, assume the sequence of vectors $\{w_t\}_{t=0}^T$ generated by the line search satisfies $\Pi^t\nabla f^t\neq 0$, \textit{i.e.}, the optimality condition has not been reached. Then $\{f(w^t)\}_{t=0}^{T}$ is a strictly decreasing sequence. Since $f$ is convex and $\Delta$ is compact, it is bounded from below on $\Delta$. Hence $f(w^T)\to f(w^*)$ as $T\to\infty$.
\end{proof}

\subsection{Proof of Theorem \ref{thm:convergence_of_linear_scheme}}\label{sec:appendix_discrete_proof}
\begin{proof}
    W.l.o.g, we assume that the sequence $\{\eta^t\}_{t=0}^T$, which satisfies the assumptions of Theorem \ref{thm:convergence_of_linear_scheme}, generates the sequence $\{w^t\}_{t=0}^T$ such that $\Pi^t\nabla f^t\neq 0$, \textit{i.e.}, the optimality condition has not been reached.
    
    By Theorem \ref{thm:appendix_relative_entropy_diff},
    \begin{align*}
        D(w^*|w^{t+1}) - D(w^*|w^{t}) \ \leq\ \eta^t \nabla f^t\cdot (w^* - w^t) + C_{\gamma_t}(\eta^t)^2 w^*\cdot(\Pi^t \nabla f^t)^2.
    \end{align*}
    By convexity of $f$, rearranging gives
    \begin{align}\label{eq:discrete_f_bound_2}
        \eta^t \big(f(w^t) - f(w^*) \big)\ \leq\ D(w^*|w^t) - D(w^*|w^{t+1}) + C_{\gamma_t} (\eta^t)^2 w^*\cdot (\Pi^t \nabla f^t)^2.
    \end{align}

    By Theorem \ref{thm:discrete_decreasing}, we have that
    \begin{align*}
        f(w^t)\ \geq\ f(w^{t+1}) + \frac{\eta^t}{2}w^t\cdot (\Pi^t \nabla f^t)^2.
    \end{align*}
    Repeatedly applying this inequality gives
    \begin{align*}
        f(w^T) + \frac{1}{2}\sum_{k=t}^{T-1}\eta^k w^k\cdot (\Pi^k\nabla f^k)^2\ \leq\ f(w^t),
    \end{align*}
    for $t \leq T - 1$. Thus, (\ref{eq:discrete_f_bound_2}) gives the bound
    \begin{align*}
        \eta^t(f(w^T) - f(w^*))\ \leq\ D(w^*|w^t) &- D(w^* | w^{t+1}) + C_{\gamma_t}(\eta^t)^2 w^*\cdot(\Pi^t\nabla f^t)^2\\
        & - \frac{1}{2}\eta^t \sum_{k=t}^{T-1}\eta^k w^k\cdot(\Pi^k\nabla f^k)^2.\notag
    \end{align*}

    Summing over time and collapsing the sum gives
    \begin{align}\label{eq:discrete_df_sum_eta}
        (f(w^T) - f(w^*))\sum_{t=0}^{T-1}\eta^t\ \leq\  D(w^*|w^0) &- D(w^*|w^T) \\
        & +\sum_{t=0}^{T-1}C_{\gamma_t}(\eta^t)^2 w^*\cdot(\Pi^t\nabla f^t)^2 \notag \\
        & -\frac{1}{2} \sum_{t=0}^{T-1}\sum_{k=t}^{T-1}\eta^k\eta^t w^k\cdot(\Pi^k\nabla f^k)^2.\notag
    \end{align}
    
    We can rewrite the last term as
    \begin{align*}
        \sum_{t=0}^{T-1}\sum_{k=t}^{T-1}\eta^k\eta^t w^k\cdot(\Pi^k\nabla f^k)^2\ =\ \sum_{t=0}^{T-1}\sum_{k=0}^{t}\eta^k\eta^t w^t\cdot(\Pi^t\nabla f^t)^2.
    \end{align*}
    Thus, the last two terms of (\ref{eq:discrete_df_sum_eta}) can be bounded by
    \begin{align*}
        S\ &:=\ \sum_{t=0}^{T-1}\eta^t \left(C_{\gamma_t}\eta^t w^*\cdot(\Pi^t\nabla f^t)^2 - \frac{w^t\cdot(\Pi^t\nabla f^t)^2}{2} \sum_{k=0}^t\eta^k\right)\\
        &\leq\ \sum_{t=0}^{T-1}\eta^t\left(C_{\gamma_t}\eta^t \max_i\, (\Pi^t \nabla f^t)^2_i - \frac{w^t\cdot(\Pi^t\nabla f^t)^2}{2} \sum_{k=0}^t\eta^k\right),
    \end{align*}
    as $w^*\in\Delta^n$. By assumption on $C_{\gamma_t}$,
    \begin{align*}
        C_{\gamma_t}\eta^t \max_i\, (\Pi^t \nabla f^t)^2_i - \frac{w^t\cdot(\Pi^t\nabla f^t)^2}{2} \sum_{k=0}^t\eta^k\ \leq\ \frac{w^t\cdot(\Pi^t\nabla f^t)^2}{2}\left(\eta^t-\sum_{k=0}^t\eta^k\right)\ \leq\ 0.
    \end{align*}
    It follows that $S\leq 0$. Thus (\ref{eq:discrete_df_sum_eta}) gives
    \begin{align*}
        f(w^T) - f(w^*)\ \leq\ \frac{D(w^*|w^0) - D(w^*|w^T)}{\sum_{t=0}^{T-1}\eta^t}. 
    \end{align*}
    Taking $w^0 = (1/n, \ldots, 1/n)$, then $D(u|w^0)\leq \log(n)$ for all $u\in\Delta^n$. Since relative entropy is non-negative,
    \begin{align*}
        f(w^T) - f(w^*)\ \leq\ \frac{\log(n)}{\sum_{t=0}^{T-1}\eta^t}.
    \end{align*}
\end{proof}

\subsection{Proof of Theorem \ref{thm:expert_advice}}\label{sec:appendix_expert_advice_proof}
\begin{proof}
Rearranging Theorem \ref{thm:appendix_relative_entropy_diff} gives
\begin{align}\label{eq:regret_bound_1}
    -\eta^t u \cdot(\Pi^t \nabla f^t)\ \leq\ D(u|w^t) - D(u|w^{t+1}) + C_{\gamma_t}(\eta^t)^2 u\cdot(\Pi^t \nabla f^t)^2.
\end{align}
Since $\nabla f^t = l^t$, we have the inequality
\begin{align*}
    -1\ \leq\ \Pi^t\nabla_i f^t\ =\ l^t_i - w^t\cdot l^t\ \leq\ 1,
\end{align*}
as $l^t_i\in[0,1]$. Thus dividing (\ref{eq:regret_bound_1}) by $\eta^t$ gives
\begin{align*}
    w^t\cdot l^t - u\cdot l^t\ \leq\ \frac{D(u|w^t) - D(u|w^{t+1})}{\eta^t} + C_{\gamma_t}\eta^t,
\end{align*}
where $\eta^t = \gamma_t \eta^{t,\max}$, for some $\gamma_t \in (0,1)$.

Since the maximum learning rate has the lower bound
\begin{align*}
    \eta^{t,\max}\ =\ \frac{1}{\max_i l_i^t - w^t\cdot l^t} \ \geq\ \frac{1}{\max_i l^t_i}\ \geq\ 1,
\end{align*}
we can take a fixed learning rate $\eta^t = \eta \in (0, 1)$. Moreover, $\gamma_t = \eta / \eta^{t,\max}\leq \eta$. Since $C_{\gamma_t}$ is an increasing function of $\gamma_t$, $C_{\gamma_t} \leq C_{\eta}$, thus giving the bound
\begin{align*}
     w^t\cdot l^t - u\cdot l^t\ \leq\ \frac{D(u|w^t) - D(u|w^{t+1})}{\eta} + C_\eta \eta.
\end{align*}

Summing over time and collapsing the sum gives the bound
\begin{align*}
    \sum_{t=1}^T w^t\cdot l^t - \sum_{t=1}^T u\cdot l^t\ &\leq\ \frac{D(u|w^1) - D(u|w^{T+1})}{\eta} + TC_\eta \eta\\
    &\leq\ \frac{D(u|w^1)}{\eta} + TC_\eta \eta\\
    &=\ \frac{D(u|w^1)}{\eta} + \frac{T\log(e^{-\eta}/(1-\eta))}{\eta },
\end{align*}
by definition of $C_\eta$. Using the inequality $\log(e^{-x}/(1-x))/x \leq x/(2(1-x))$ for $0\leq x\leq 1$,
\begin{align*}
    \sum_{t=1}^T w^t\cdot l^t - \sum_{t=1}^T u\cdot l^t\ \leq\ \frac{D(u|w^1)}{\eta} + \frac{T\eta}{2(1-\eta)}.
\end{align*}

Let $w^1=(1/N, \ldots, 1/N)$, then $D(u|w^1)\leq \log(N)$ for all $u\in\Delta^N$. Thus giving the desired bound
\begin{align*}
    \sum_{t=1}^T w^t\cdot l^t - \sum_{t=1}^T u\cdot l^t\ \leq\ \frac{\log(N)}{\eta} + \frac{T\eta}{2(1-\eta)}.
\end{align*}
The right side of this inequality is minimized when $\eta = \frac{\sqrt{2\log (N)}}{\sqrt{2\log (N)}+\sqrt{T}} < 1$. Upon substitution gives the bound
\begin{align*}
    \sum_{t=1}^T w^t\cdot l^t - \sum_{t=1}^T u\cdot l^t\ \leq\ \sqrt{2 T \log (N)} + \log (N).
\end{align*}
\end{proof}

\subsection{Proof of Theorem \ref{thm:portfolio_cauchy_simplex}}\label{sec:appendix_universal_portfolio_proof}
\begin{proof}
Rearranging Theorem \ref{thm:appendix_relative_entropy_diff} gives
\begin{align*}
    -\eta^t u \cdot(\Pi^t \nabla f^t)\ \leq\ D(u|w^t) - D(u|w^{t+1}) + C_{\gamma_t}(\eta^t)^2 u\cdot(\Pi^t \nabla f^t)^2.
\end{align*}
Since $\nabla f = -x^t/(w^t\cdot x^t)$, we have the inequality
\begin{align*}
    -\frac{1}{a}\ \leq\ \Pi\nabla_i f = 1 - x_i^t/(w^t\cdot x^t)\ \leq\ 1,
\end{align*}
as $x^t_i\in[{a},1]$. Thus diving by $\eta^t$ gives the bound
\begin{align*}
    \frac{u\cdot x^t}{w^t\cdot x^t} - 1\ \leq\ \frac{D(u|w^t) - D(u|w^{t+1})}{\eta^t} + \frac{C_\gamma \eta^t}{a^2}.
\end{align*}
Using the inequality $e^x-1\geq x$ for all $x$ gives
\begin{align*}
    \log\left(\frac{u\cdot x^t}{w^t\cdot x^t}\right)\ \leq\ \frac{D(u|w^t) - D(u|w^{t+1})}{\eta^t} + \frac{C_\gamma \eta^t}{a^2}.
\end{align*}

Since the maximum learning rate has the lower bound
\begin{align*}
    \eta^{t,\max}\ =\ \frac{1}{\max_i(1 - x^t_i/w^t\cdot x^t)}\ =\ \frac{1}{1 - \min_i (x^t_i/w^t\cdot x^t)}\ \geq\ 1,
\end{align*}
we can take a fixed learning rate $\eta^t = \eta$. 

Following the steps from Theorem \ref{thm:expert_advice} gives the bound
\begin{align*}
    \sum_{t=1}^T\log(u\cdot l^t) - \sum_{t=1}^T \log(w^t\cdot l^t)\ \leq\ \frac{D(u|w^1)}{\eta} + \frac{T\eta}{2 a^2(1-\eta)}.
\end{align*}
Taking $w^1=(1/N, \ldots, 1/N)$ and minimizing the right-hand side of the inequality w.r.t. $\eta$ gives $\eta = \frac{a \sqrt{2 \log (N)}}{a \sqrt{2 \log (N)}+\sqrt{T}}$. Thus giving the bound
\begin{align*}
    \sum_{t=1}^T\log(u\cdot l^t) - \sum_{t=1}^T \log(w^t\cdot l^t)\ \leq\ \frac{\sqrt{2 T \log (N)}}{a}+\log (N).
\end{align*}
\end{proof}

\section{Karush-Kuhn-Tucker Conditions}\label{sec:KKT}
The Karush-Kuhn-Tucker (KKT) Conditions are first-order conditions that are necessary but insufficient for optimality in constrained optimization problems. For convex problems, it becomes a sufficient condition for optimality. 
{We give a brief overview of the KKT conditions here, and for a full treatment of the subject, we suggest \cite{Kochenderfer2019}.}

Consider a general constrained optimization problem
\begin{align}\label{eq:general_constrained_problem}
\min_w f(w)\qquad \textrm{s.t.}\qquad g_i(w)\ \leq\ 0\quad \textrm{and}\quad h_j(w)\ =\ 0,
\end{align}
for $i=1,...,n$ and $j=1,...,m$. The (primal) Lagrangian is defined as 
\begin{align*}
\mathcal{L}(w, \alpha, \beta)\ =\ f(w) + \sum_i\alpha_i g_i(w) + \sum_j\beta_j h_j(w).
\end{align*}

Consider the new optimization problem
\begin{align*}
\theta(w)\ =\ \max_{\alpha, \beta;\ \alpha_i\geq0}\mathcal{L}(w, \alpha, \beta).
\end{align*}
Note that
\begin{align*}
\theta(w)\ =\ \begin{cases}
    f(w)\qquad &\textrm{if $w$ satisfies the constraints}\\
    \infty \qquad &\textrm{otherwise}.
\end{cases}
\end{align*}
Hence, $\alpha_{i}$ and $\beta_j$ are slack variables that render a given Lagrangian variation equation irrelevant when violated. 

To solve (\ref{eq:general_constrained_problem}), we can instead consider the new optimization problem
\begin{align*}
\min_w \theta(w)\ =\ \min_w \max_{\alpha, \beta;\ \alpha_i\geq0}\mathcal{L}(w, \alpha, \beta).
\end{align*}

Assume $f$ and $g_i$ are convex, $h_j$ is affine, and the constraints are feasible. A solution $(w^*, \alpha^*, \beta^*)$ is an optimal solution to (\ref{eq:general_constrained_problem}) if the following conditions, known as the KKT conditions, are satisfied:
\begin{align*}
\frac{\partial}{\partial w_i}\mathcal{L}(w^*, \alpha^*, \beta^*)\ &=\ 0\qquad \text{(Stationarity)}\\
\alpha_i^*g_i(w^*)\ &=\ 0\qquad \text{(Complementary Slackness)}\\
g_i(w^*)\ \leq \ 0,\ h_j(w)\ &=\ 0 \qquad \text{(Primal Feasibility)}\\
\alpha_i^*\ & \geq \ 0 \qquad \text{(Dual Feasibility)},
\end{align*}
for all $i$ and $j$.

When the constraint is a simplex, the Lagrangian becomes
\begin{align*}
    \mathcal{L}(w,\alpha,\beta)\ =\ f(w) - \sum_i \alpha_iw_i + \beta\left(\sum_i w_i - 1\right).
\end{align*}
Thus stationarity gives
\begin{align*}
    \frac{\partial}{\partial w_i}\mathcal{L}\ =\ \grad_i f - \alpha_i + \beta = 0.
\end{align*}
Let $Q=\{i : w_i = 0\}$ be the active set and $S = \{i : w_i > 0\}$ be the support. The complementary slackness requires $\alpha_i=0$ for $i\in S$, so stationarity gives $\beta = \grad_i f$, \textit{i.e.} constant on the support. The active set's dual feasibility and stationarity conditions thus require $\alpha_i = \beta + \grad_i f \geq 0$.

\end{document}